\documentclass[10pt]{amsart}

  \baselineskip 1 mm

\usepackage{amsmath,amsthm,amsfonts,latexsym,amsopn,verbatim,amscd,amssymb}
\usepackage{hyperref}

\theoremstyle{plain}
\newtheorem{theorem}{Theorem}[section]
\newtheorem{lemma}[theorem]{Lemma}
\newtheorem{proposition}[theorem]{Proposition}
\newtheorem{corollary}[theorem]{Corollary}
\theoremstyle{definition}

\numberwithin{equation}{section}

\DeclareMathOperator{\spec}{Spec}
\DeclareMathOperator{\L-spec}{L-Spec}
\DeclareMathOperator{\Q-spec}{Q-Spec}
\DeclareMathOperator{\cent}{Cent}

\newcommand{\bnum}{\begin{enumerate}}
\newcommand{\enum}{\end{enumerate}}

\begin{document}

\title{Various energies of some super integral groups}
\author[P. Dutta and R. K. Nath]{Parama Dutta  and Rajat Kanti Nath$^*$}
\address{Department of Mathematical Sciences, Tezpur
University,  Napaam-784028, Sonitpur, Assam, India.
}
\email{parama@gonitsora.com and rajatkantinath@yahoo.com$^*$}


%\date{}
%\maketitle
%\begin{center}
%Sunil Kumar Prajapati\\
%\small{ Institute of Mathematics\\
%The Hebrew University,
%Jerusalem 91904, Israel. \\
%Email: skprajapati.iitd@gmail.com}
%\end{center}
%
%\begin{center}
%Rajat Kanti Nath\\
%\small{ Department of Mathematical Sciences,\\ Tezpur
%University,  Napaam-784028, Sonitpur, Assam, India.\\
%Email: rajatkantinath@yahoo.com}
%\end{center}

\thanks{*Corresponding author}

%\smallskip
%
%
%\noindent {\small{\textbf{Abstract:}
\begin{abstract}
In  this paper, we obtain energy,  Laplacian energy and signless Laplacian energy of the commuting graphs of some families of finite non-abelian groups.
\end{abstract}

\maketitle

%\bigskip
%
%\noindent \small{\textbf{\textit{Key words:}}  Word map, Commutator, Irreducible character, finite group.}
%%\smallskip
%
%\noindent \small{\textbf{\textit{2010 Mathematics Subject Classification:}} 20F70, 20C15. }

\section{Introduction} \label{S:intro}

%Let $V({\mathcal{G}})$ denote the set of vertices of the graph ${\mathcal{G}}$.
Let $A({\mathcal{G}})$ and $D({\mathcal{G}})$ denote the adjacency matrix  and degree matrix of a graph  ${\mathcal{G}}$ respectively. Then the Laplacian matrix and   signless Laplacian matrix of ${\mathcal{G}}$ are given by  $L({\mathcal{G}})  = D({\mathcal{G}}) - A({\mathcal{G}})$ and    $Q({\mathcal{G}}) = D({\mathcal{G}}) + A({\mathcal{G}})$ respectively. We write  $\spec({\mathcal{G}})$, $\L-spec({\mathcal{G}})$ and $\Q-spec({\mathcal{G}})$ to denote the spectrum, Laplacian spectrum and Signless Laplacian spectrum of ${\mathcal{G}}$. Also, $\spec({\mathcal{G}}) = \{\alpha_1^{a_1}, \alpha_2^{a_2}, \dots, \alpha_l^{a_l}\}$, $\L-spec({\mathcal{G}}) = \{\beta_1^{b_1}, \beta_2^{b_2}, \dots, \beta_m^{b_m}\}$ and $\Q-spec({\mathcal{G}}) = \{\gamma_1^{c_1}, \gamma_2^{c_2}, \dots, \gamma_n^{c_n}\}$ where $\alpha_1,  \alpha_2, \dots, \alpha_n$ are the eigenvalues of   $A({\mathcal{G}})$ with multiplicities $a_1, a_2, \dots, a_l$;   $\beta_1,  \beta_2, \dots, \beta_m$ are the eigenvalues of  $L({\mathcal{G}})$ with multiplicities $b_1, b_2, \dots, b_m$ and  $\gamma_1, \gamma_2, \dots, \gamma_n$ are the eigenvalues of   $Q({\mathcal{G}})$ with multiplicities $c_1, c_2, \dots, c_n$ respectively.  Harary and Schwenk \cite{hS74} introduced the concept of  integral graphs in 1974. A graph $\mathcal{G}$ is called integral or L-integral or Q-integral according as $\spec({\mathcal{G}})$ or $\L-spec({\mathcal{G}})$ or $\Q-spec({\mathcal{G}})$ contains only integers. One may conf. \cite{anb09,bCrS03, Abreu08,Kirkland07,Merries94,Simic08} for various results of these graphs.
% Integral graphs are studied in \cite{anb09,bCrS03},   L-integral and Q-integral  graphs are studied in \cite{Abreu08,Kirkland07,Merries94,Simic08}.

Depending on various spectra of a graph there are various energies called \textit{energy}, \textit{Laplacian energy} and \textit{signless Laplacian energy} denoted by $E(\mathcal{G}),
 LE(\mathcal{G})$ and $LE^+(\mathcal{G})$ respectively. These energies are  defined as follows:

\begin{equation}\label{energy}
E(\mathcal{G}) = \sum_{\lambda \in \spec({\mathcal{G}})}|\lambda|,
\end{equation}

\begin{equation}\label{Lenergy}
LE(\mathcal{G}) = \sum_{\mu \in \L-spec({\mathcal{G}})}\left|\mu - \frac{2|e(\mathcal{G}|)}{|v(\mathcal{G})|}\right|,
\end{equation}
 and
\begin{equation}\label{Qenergy}
LE^+(\mathcal{G}) = \sum_{\nu \in \Q-spec({\mathcal{G}})}\left|\nu - \frac{2|e(\mathcal{G}|)}{|v(\mathcal{G})|}\right|,
\end{equation}
where $v(\mathcal{G})$ and $e(\mathcal{G})$ denotes the set of vertices and edges of $\mathcal{G}$ respectively.
% Write few lines on these energies and include some references........................

%A graph $\mathcal{G}$ is called  integral or L-integral or  Q-integral according as  $\spec({\mathcal{G}})$ or $\L-spec({\mathcal{G}})$ or $\Q-spec({\mathcal{G}})$ contains only integers.
%The notion of integral graph was introduced by  Harary and  Schwenk \cite{hS74} in the year 1974. A very impressive survey on integral graphs can be found in \cite{bCrS03}.
% Ahmadi et al. noted  that   integral graphs have some interests for designing the network topology of perfect state transfer networks, see  \cite{anb09} and the references there in. L-integral graphs are also  studied extensively over the years while $Q$-integral graphs are not studied much. One may conf. \cite{Abreu08, Simic07,Kirkland07,  Merries94,Simic08} and some of the references in \cite{Kirkland07}  for several interesting results of these graphs.

Let $G$ be a finite non-abelian group with center $Z(G)$. The commuting graph of  $G$, denoted by $\Gamma_G$, is a simple undirected graph whose vertex set is $G\setminus Z(G)$, and two distinct vertices $x$ and $y$ are adjacent if and only if $xy = yx$.
% The energy, Laplacian energy  and  signless Laplacian energy of   $G$ is the energy, Laplacian energy  and  signless Laplacian energy of $\Gamma_G$.
  Various   aspects of commuting graphs of different finite groups can be found in  \cite{amr06,iJ07,mP13,par13}. In \cite{Dutta16,DN16,dn016}, Dutta and Nath have computed various spectra of $\Gamma_G$ for different families of finite groups. A finite non-abelian group is called super integral if $\spec(\Gamma_G), \L-spec(\Gamma_G)$  and $\Q-spec(\Gamma_G)$ contain only integers. Various examples of  super integral groups can be found in \cite{dn016}. The energy, Laplacian energy  and  signless Laplacian energy of $\Gamma_G$ are called    energy, Laplacian energy  and  signless Laplacian energy of   $G$ respectively.    In this paper, we compute various energies of $G$ for some families of super integral groups. It may be mentioned here that the Laplacian energy of non-commuting graphs of various finite non-abelian groups are computed in \cite{Dutta_Nath_2016}.

\section{Some Computations}
% It is well-known that  $\L-spec(K_n) = \{0^1, n^{n - 1}\}$ and $\Q-spec(K_n) = \{(2n -2)^1, (n - 2)^{n - 1}\}$, where $K_n$ denotes the complete graph on $n$ vertices. Further, we have the following theorem.

%\begin{theorem}\label{prethm1}
%If $\mathcal{G} = l_1K_{m_1}\sqcup l_2K_{m_2}\sqcup\cdots  \sqcup l_kK_{m_k}$, where $l_iK_{m_i}$ denotes the disjoint union of $l_i$ copies of  $K_{m_i}$ for $1 \leq i \leq k$, then
%\[
%\spec({\mathcal{G}}) = \left\{(-1)^{\sum_{i = 1}^kl_i(m_i - 1)},\, (m_1 - 1)^{l_1},\, (m_2 - 1)^{l_2},\, \dots,\, (m_k - 1)^{l_k}\right\}
%\]
%\[
%\L-spec({\mathcal{G}}) = \left\{0^{\sum_{i=1}^k l_i}, m_1^{l_1(m_1 - 1)}, m_2^{l_2(m_2 - 1)}, \dots, m_k^{l_k(m_k - 1)}\right\}
%\] and
%\begin{align*}
%\Q-spec({\mathcal{G}}) = \{&(2m_1 - 2)^{l_1}, (m_1 - 2)^{l_1(m_1 - 1)}, (2m_2 - 2)^{l_2}, (m_2 - 2)^{l_2(m_2 - 1)}, \\
%&\dots, (2m_k - 2)^{l_k}, (m_k - 2)^{l_k(m_k - 1)}\}.
%\end{align*}
%\end{theorem}

In this section, we compute    various energies of  the  commuting graphs of some  families of finite non-abelian groups.  We begin with some families of groups whose central factors are some well-known  groups.
\begin{theorem} \label{order-20}
Let $G$ be a finite group and $\frac{G}{Z(G)} \cong Sz(2)$, where $Sz(2)$ is the Suzuki group presented by $\langle a, b : a^5 = b^4 = 1, b^{-1}ab = a^2 \rangle$. Then
\[
E(\Gamma_G) =  38|Z(G)| - 12,\quad
LE(\Gamma_G) =  \begin{cases}
\frac{732|Z(G)| - 228}{19} & \text{ if } |Z(G)| \leq 4\\
\frac{120|Z(G)|^2 + 122|Z(G)| - 38}{19}  & \text{ if } |Z(G)| > 4
\end{cases} \text{\quad and}
\]
\[
LE^+(\Gamma_G) = \begin{cases}
\frac{712|Z(G)| - 228 }{19} & \text{ if } |Z(G)| = 1\\
\frac{120|Z(G)|^2 - 530|Z(G)| - 190 }{19}  & \text{ if } |Z(G)| > 1.
\end{cases}
\]
\end{theorem}
\begin{proof}

By \cite[Theorem 2]{Dutta16}, we have
\[
\spec(\Gamma_G) = \{(-1)^{19|Z(G)| - 6}, (4|Z(G)| - 1)^1,  (3|Z(G)| - 1)^5\}.
\]
Therefore, by \eqref{energy}, we have
\begin{align*}
E(\Gamma_G) = & 19|Z(G)| - 6 + 4|Z(G)| - 1 + 5(3|Z(G)| - 1)
=  38|Z(G)| - 12.
\end{align*}
Also, $v(\Gamma_G) = 19|Z(G)|$ and $e(\Gamma_G) = \frac{4|Z(G)|(4|Z(G)| - 1) + 15|Z(G)|(3|Z(G)| - 1)}{2}$ as $\Gamma_G = K_{4|Z(G)|}\sqcup 5K_{3|Z(G)|}$. Therefore,
\begin{align*}
\frac{2|e(\Gamma_G)|}{|v(\Gamma_G)|} = & \frac{4|Z(G)|(4|Z(G)| - 1) + 15|Z(G)|(3|Z(G)| - 1)}{19|Z(G)|}\\
= & \frac{16|Z(G)| - 4 + 45|Z(G)| - 15}{19} = \frac{61|Z(G)| - 19}{19}.
\end{align*}
Note that for any two  integers $r, s$, we have
\begin{equation}\label{eqsuzuki1}
r|Z(G)| + s -  \frac{2|e(\Gamma_G)|}{|v(\Gamma_G)|} = \frac{(19r - 61)|Z(G)| + 19(s + 1)}{19}.
\end{equation}

By \cite[Theorem 2.2]{dn016},  we have
\[
\L-spec(\Gamma_G) =  \{0^6, (4|Z(G)|)^{4|Z(G)| - 1}, (3|Z(G)|)^{15|Z(G)| - 5}\}.
\]
%\[
%\Q-spec(\Gamma_G) = \{(8|Z(G)| - 2)^1, (4|Z(G)| - 2)^{4|Z(G)| - 1}, (6|Z(G)| - 2)^5, (3|Z(G)| - 2)^{15|Z(G)| - 5}\}.
%\]
Using \eqref{eqsuzuki1}, we have $\left|0 - \frac{2|e(\mathcal{G}|)}{|v(\mathcal{G})|}\right| = \frac{61|Z(G)| - 19}{19}$, $
\left|4|Z(G)|  - \frac{2|e(\Gamma_G)|}{|v(\Gamma_G)|}\right| =  \frac{15|Z(G)| + 19}{19}$ and
\[
\left|3|Z(G)| - \frac{2|e(\Gamma_G)|}{|v(\Gamma_G)|}\right| = \left|\frac{-4|Z(G)| + 19}{19}\right| = \begin{cases}
\frac{-4|Z(G)| + 19}{19} & \text{ if } |Z(G)| \leq 4\\
\frac{4|Z(G)| - 19}{19}  & \text{ if } |Z(G)| > 4.
\end{cases}
\]
Therefore, if $|Z(G)| \leq 4$, then by \eqref{Lenergy} we have
\begin{align*}
LE(\Gamma_G) = & \frac{366|Z(G)| - 114}{19} + \frac{(4|Z(G)| - 1)(15|Z(G)| + 19)}{19} + \frac{(15|Z(G)| - 5)(-4|Z(G)| + 19)}{19} \\
= & \frac{366|Z(G)| - 114 + 60|Z(G)|^2 + 61|Z(G)| - 19 -60|Z(G)|^2 + 305|Z(G)| -95}{19}\\
= & \frac{732|Z(G)| - 228}{19}.
\end{align*}
If $|Z(G)| > 4$, then by \eqref{Lenergy} we have
\begin{align*}
LE(\Gamma_G) = & \frac{366|Z(G)| - 114}{19} + \frac{(4|Z(G)| - 1)(15|Z(G)| + 19)}{19} + \frac{(15|Z(G)| - 5)(4|Z(G)| - 19)}{19} \\
= & \frac{366|Z(G)| - 114 + 60|Z(G)|^2 + 61|Z(G)| - 19 + 60|Z(G)|^2 - 305|Z(G)| +95}{19}\\
= & \frac{120|Z(G)|^2 + 122|Z(G)| - 38}{19}.
\end{align*}

By \cite[Theorem 2.2]{dn016}  we also have
\[
\Q-spec(\Gamma_G) = \{(8|Z(G)| - 2)^1, (4|Z(G)| - 2)^{4|Z(G)| - 1}, (6|Z(G)| - 2)^5, (3|Z(G)| - 2)^{15|Z(G)| - 5}\}.
\]
Now, using \eqref{eqsuzuki1} we have

\noindent $\left|8|Z(G)| - 2 - \frac{2|\Gamma_G)|}{|v(\Gamma_G)|}\right| = \frac{91|Z(G)| - 19}{19}$, $
\left|4|Z(G)| - 2  - \frac{2|e(\Gamma_G)|}{|v(\Gamma_G)|}\right| = \begin{cases}
\frac{-15|Z(G)| +  19}{19} & \text{ if } |Z(G)| = 1\\
\frac{15|Z(G)| -  19}{19} & \text{ if } |Z(G)| > 1,
\end{cases}$

\noindent  $
\left|6|Z(G)| - 2  - \frac{2|e(\Gamma_G)|}{|v(\Gamma_G)|}\right| =  \frac{53|Z(G)| - 19}{19}$ and $\left|3|Z(G)| - 2  - \frac{2|e(\Gamma_G)|}{|v(\Gamma_G)|}\right| =  \frac{4|Z(G)| + 19}{19}$. Hence, if $|Z(G)|$ $= 1$, then  by \eqref{Qenergy} we have
\begin{align*}
LE^+({\mathcal{G}}) = & \frac{91|Z(G)| - 19}{19} + \frac{(4|Z(G)| - 1)(-15|Z(G)| +  19)}{19} +  \frac{265|Z(G)| - 95}{19}\\
& + \frac{(15|Z(G)| - 5)(4|Z(G)| + 19)}{19}\\
%= & \frac{91|Z(G)| - 19 - 60|Z(G)|^2 + 91|Z(G)| - 19 +  265|Z(G)| - 95 + 60|Z(G)|^2 + 265|%Z(G)| - 95}{19}\\
=  & \frac{712|Z(G)| - 228 }{19}.
\end{align*}
If $|Z(G)| > 1$, then by \eqref{Qenergy} we have
\begin{align*}
LE^+({\mathcal{G}}) = & \frac{91|Z(G)| - 19}{19} + \frac{(4|Z(G)| - 1)(15|Z(G)| -  19)}{19} +  \frac{5(53|Z(G)| - 19)}{19}\\
&  + \frac{(15|Z(G)| - 5)(4|Z(G)| + 19)}{19}\\
%= & \frac{91|Z(G)| - 19 + 60|Z(G)|^2 - 91|Z(G)| + 19 +  265|Z(G)| - 95 + 60|Z(G)|^2 + 265|%Z(G)| - 95}{19}\\
=  & \frac{120|Z(G)|^2 - 530|Z(G)| - 190 }{19}.
\end{align*}
\end{proof}

\begin{theorem}\label{main2}
Let $G$ be a finite group such that $\frac{G}{Z(G)} \cong {\mathbb{Z}}_p \times {\mathbb{Z}}_p$, where $p$ is a prime integer. Then
\[
E(\Gamma_G) =  LE(\Gamma_G) = LE^+(\Gamma_G) = 2(p^2 - 1)|Z(G)| - 2(p + 1).
\]
\end{theorem}
\begin{proof}
By \cite[Theorem 2.1]{DN16} we have
\[
\spec(\Gamma_G) = \{(-1)^{(p^2 - 1)|Z(G)| - p - 1}, ((p - 1)|Z(G)| - 1)^{p + 1}\}.
\]
Therefore, by \eqref{energy}, we have
\[
E(\Gamma_G) = (p^2 - 1)|Z(G)| - p - 1 + (p + 1)((p - 1)|Z(G)| - 1) = 2(p^2 - 1)|Z(G)| - 2(p + 1).
\]
We have, $|v(\Gamma_G)| = (p^2 - 1)|Z(G)|$ and $\Gamma_G = (p + 1) K_{(p - 1)|Z(G)|}$. Therefore, $2|e(\Gamma_G)| = (p^2 - 1)|Z(G)|((p - 1)|Z(G)| - 1)$ and so
\[
\frac{2|e(\Gamma_G)|}{|v(\Gamma_G)|} = (p - 1)|Z(G)| - 1.
\]

By \cite[Theorem 2.3]{dn016}, we have
\[
\L-spec(\Gamma_G) = \{0^{p +1}, ((p - 1)|Z(G)|)^{(p^2 - 1)|Z(G)| - p - 1}\}.
\]
Now, $\left|0 - \frac{2|e(\Gamma_G)|}{|v(\Gamma_G)|}\right| = (p - 1)|Z(G)| - 1$ and  $\left|(p - 1)|Z(G)| - \frac{2|e(\Gamma_G)|}{|v(\Gamma_G)|}\right| =  1$. Hence, by \eqref{Lenergy}, we have
\[
LE(\Gamma_G) = (p + 1)((p - 1)|Z(G)| - 1) + (p^2 - 1)|Z(G)| - p - 1 =  2(p^2 - 1)|Z(G)| - 2(p + 1).
\]

By \cite[Theorem 2.3]{dn016}, we also have
\[
\Q-spec(\Gamma_G) = \{(2(p - 1)|Z(G)| - 2)^{p + 1}, ((p - 1)|Z(G)| - 2)^{(p^2 - 1)|Z(G)| - p - 1}\}.
\]
Now, $\left|2(p - 1)|Z(G)| - 2 - \frac{2|e(\Gamma_G)|}{|v(\Gamma_G)|}\right| = (p - 1)|Z(G)| - 1$ and  $\left|(p - 1)|Z(G)| - 2 - \frac{2|e(\Gamma_G)|}{|v(\Gamma_G)|}\right| =  1$. Hence, by \eqref{Qenergy}, we have
\[
LE^+(\Gamma_G) = (p + 1)((p - 1)|Z(G)| - 1) + (p^2 - 1)|Z(G)| - p - 1 = 2(p^2 - 1)|Z(G)| - 2(p + 1).
\]
\end{proof}
\noindent As a corollary we have the following result.
\begin{corollary}
Let $G$ be a non-abelian group of order $p^3$, for any prime $p$, then
\[
E(\Gamma_G) =  LE(\Gamma_G) = LE^+(\Gamma_G) = 2p^3 - 4p  - 2.
\]
\end{corollary}

\begin{proof}
Note that $|Z(G)| = p$ and  $\frac{G}{Z(G)} \cong {\mathbb{Z}}_p \times {\mathbb{Z}}_p$. Hence the  result follows from Theorem \ref{main2}.
\end{proof}

% The following theorem shows that $G$ is commuting integral if the central factor of $G$ is isomorphic to the dihedral group $D_{2m} = \{a, b : a^{m} = b^2 = 1, bab^{-1} = a^{-1}\}$.
\begin{theorem}\label{main4}
Let $G$ be a finite group such that $\frac{G}{Z(G)} \cong D_{2m}$, for $m \geq 2$. Then
\begin{enumerate}
\item $E(\Gamma_G) = (4m - 2)|Z(G)| - 2(m + 1)$.
\item If $m = 2$;   $m = 3$ and $|Z(G)|=1,2$;
$m = 4$  and $|Z(G)| = 1$ then
\[
LE(\Gamma_G) = \frac{(2m^3 + 2)|Z(G)| - 4m^2 - 2m + 2}{2m - 1}.
\]
\item If $m \geq 3$ and $|Z(G)| \geq 3$;  or $m = 4$ and $|Z(G)| \geq 2$; or $m \geq 5$ then
\[
LE(\Gamma_G) = \frac{(2m^3 - 6m^2 + 4m)|Z(G)|^2 + (2m^2 - 2m + 2)|Z(G)| - 4m + 2}{2m - 1}.
\]
\item If $m = 2$ then $LE^+(\Gamma_G) = 6|Z(G)| - 6.$
\item If $m = 3$ and $|Z(G)| = 1$ then $LE^+(\Gamma_G) = \frac{16}{5}$.
\item If $m = 3$ and $|Z(G)| \geq 2$ then $LE^+(\Gamma_G) = \frac{12|Z(G)|^2 + 18|Z(G)| - 30}{5}$.
\item If $m = 4$ and $|Z(G)| \leq 6$ then $LE^+(\Gamma_G) = \frac{48|Z(G)|^2}{7}$.
\item If $m = 4$ and $|Z(G)| > 6$ then $LE^+(\Gamma_G) = \frac{48|Z(G)|^2 + 8|Z(G)| - 56}{7}$.
\item If $m \geq 5$  then $LE^+(\Gamma_G) = \frac{(2m^3 - 6m^2 + 4m)|Z(G)|^2 + (m^3 - 7m^2 + 4m)|Z(G)| - 2m^2 + 3m - 1}{2m - 1}$.
\end{enumerate}

%\[
%E(\Gamma_G) = (4m - 2)|Z(G)| - 2(m + 1),
%\]
%\[
%LE(\Gamma_G) = \begin{cases}
%\frac{(m^3 + 4m^2 - 5m + 4)|Z(G)| - (2m^2 + 5m - 1)}{2m - 1}    \hspace{1.7cm}\text{ if $m = 2$;} &  \text{ $m = 3$ and $1 \leq |Z(G)| \leq 2$;}\\
%\hspace{7cm} $m = 4$ & \text{ and $|Z(G)| = 1$; or $m \geq 5$}\\
%\frac{(2m^3 - 6m^2 + 4m)|Z(G)|^2 + (m^3 - 2m^2 + 3m)|Z(G)| - (2m^2 + m + 1)}{2m - 1} & \text{ if $m \geq 3$ and $|Z(G)| \geq 3$;}\\
%&\text{ or $m = 4$ and $|Z(G)| \geq 2$}
%\end{cases}
%\]
%and
%$
%LE^+(\Gamma_G) =\begin{cases}
% & \text{ if } \\
%& \text{ if }
%\end{cases}
%$
\end{theorem}

\begin{proof}
By \cite[Theorem 2.5]{DN16}, we have
\[
\spec(\Gamma_G) = \{(-1)^{(2m - 1)|Z(G)| - m - 1}, (|Z(G)| - 1)^m, ((m - 1)|Z(G)| - 1)^1\}.
\]
Therefore, by \eqref{Lenergy}, we have
\[
E(\Gamma_G) = (2m - 1)|Z(G)| - m - 1 + m(|Z(G)| - 1) + (m - 1)|Z(G)| - 1 = (4m - 2)|Z(G)| - 2(m + 1).
\]
Note that $|v(\Gamma_G)| = (2m - 1)|Z(G)|$ and $2|e(\Gamma_G)| = (m - 1)|Z(G)|((m - 1)|Z(G)|  - 1) + m|Z(G)|(|Z(G)|$  $- 1)$ since $\Gamma_G = K_{(m - 1)|Z(G)|} \sqcup m K_{|Z(G)|}$. Therefore,
\[
\frac{2|e(\Gamma_G)|}{|v(\Gamma_G)|} = \frac{(m - 1)((m - 1)|Z(G)|  - 1) + m(|Z(G)| - 1)}{2m - 1} = \frac{(m^2 - m + 1)|Z(G)| - 2m + 1}{2m - 1}.
\]
Note that for any two  integers $r, s$ we have
\begin{equation}\label{eqdihedral}
r|Z(G)| + s -  \frac{2|e(\Gamma_G)|}{|v(\Gamma_G)|} = \frac{((2r + 1)m - m^2 - r - 1)|Z(G)| + 2m(s + 1) - s - 1}{2m - 1}.
\end{equation}

By \cite[Theorem 2.5]{dn016}, we have
\[
\L-spec(\Gamma_G) = \{0^{m + 1}, ((m - 1)|Z(G)|)^{(m - 1)|Z(G)| - 1}, (|Z(G)|)^{m(|Z(G)| - 1)}\}
\]
Therefore, using \eqref{eqdihedral}, we have

$\left|0 - \frac{2|e(\Gamma_G)|}{|v(\Gamma_G)|}\right| = \frac{(m^2 - m + 1)|Z(G)| - 2m + 1}{2m - 1}$, $\left|(m - 1)|Z(G)| - \frac{2|e(\Gamma_G)|}{|v(\Gamma_G)|}\right| = \frac{(m^2 - 2m)|Z(G)| + 2m - 1}{2m - 1}$ and

$\left||Z(G)| - \frac{2|e(\Gamma_G)|}{|v(\Gamma_G)|}\right| = \begin{cases}
\frac{(3m - m^2 - 2)|Z(G)| + 2m - 1}{2m - 1} & \text{ if $m = 2$; or $m = 3$ and $|Z(G)| = 1, 2$}; or\\
& \text{ $m = 4$ and $|Z(G)| = 1$}\\
\frac{(-3m + m^2 + 2)|Z(G)| - 2m + 1}{2m - 1} & \text{ if $m = 3$ and $|Z(G)| \geq 3$; or}\\
& \text{ $m = 4$ and $|Z(G)| \geq 2$; or $m \geq 5$}.
\end{cases}$
\\
Therefore, if $m = 2$; or $m = 3$ and $|Z(G)| = 1$ or $2$; or  $m = 4$ and $|Z(G)| = 1$ then by \eqref{Lenergy}, we have
\begin{align*}
LE(&\Gamma_G) \\
=& \frac{(m + 1)((m^2 - m + 1)|Z(G)| - 2m + 1)}{2m - 1} + \frac{((m - 1)|Z(G)| - 1)((m^2 - 2m)|Z(G)| + 2m - 1)}{2m - 1}\\
& + \frac{(m(|Z(G)| - 1))((3m - m^2 - 2)|Z(G)| + 2m - 1)}{2m - 1}\\
%=& \frac{(m^3 - 3m^2 + 2m)|Z(G)|^2 + (m^3 + m^2 - m + 2)|Z(G)| - (2m^2 + 3m)}{2m - 1}\\
%& + \frac{(3m^2 - m^3 - 2m)|Z(G)|^2 + (3m^2 -4m + 2)|Z(G)| - (2m - 1)}{2m - 1}\\
= &\frac{(2m^3 + 2)|Z(G)| - 4m^2 - 2m + 2}{2m - 1}.
\end{align*}
\noindent
If $m \geq 3$ and $|Z(G)| = 3$; or $m = 4$ and $|Z(G)| \geq 2$; or $m \geq 5$ then by \eqref{Lenergy} we have
\begin{align*}
LE(&\Gamma_G)\\
= & \frac{(m + 1)((m^2 - m + 1)|Z(G)| - 2m + 1)}{2m - 1} + \frac{((m - 1)|Z(G)| - 1)((m^2 - 2m)|Z(G)| + 2m - 1)}{2m - 1}\\
 & + \frac{(m(|Z(G)| - 1))((-3m + m^2 + 2)|Z(G)| - 2m + 1)}{2m - 1}\\
%=& \frac{(m^3 - 3m^2 + 2m)|Z(G)|^2 + (m^3 + m^2 - m + 2)|Z(G)| - (2m^2 + 3m)}{2m - 1}\\
%& + \frac{(-3m^2 + m^3 + 2m)|Z(G)|^2 - (3m^2 -4m + 2)|Z(G)| + (2m - 1)}{2m - 1}\\
= & \frac{(2m^3 - 6m^2 + 4m)|Z(G)|^2 + (2m^2 - 2m + 2)|Z(G)| - 4m + 2}{2m - 1}.
\end{align*}

By \cite[Theorem 2.5]{dn016}, we also have
\begin{align*}
\Q-spec(\Gamma_G) =  &\{(2(m - 1)|Z(G)| - 2)^1, ((m - 1)|Z(G)| - 2)^{(m - 1)|Z(G)| - 1},\\
& \hspace{4cm} (2|Z(G)| - 2)^m, (|Z(G)| - 2)^{m(|Z(G)| - 1)}\}.
\end{align*}
Now, using \eqref{eqdihedral}, we have

$\left|2(m - 1)|Z(G)| - 2 - \frac{2|e(\Gamma_G)|}{|v(\Gamma_G)|}\right| = \frac{(3m^2 - 5m + 1)|Z(G)| - 2m + 1}{2m - 1}$,

$\left|(m - 1)|Z(G)| - 2 - \frac{2|e(\Gamma_G)|}{|v(\Gamma_G)|}\right| = \begin{cases}
\frac{(m^2 - 2m)|Z(G)| - 2m + 1}{2m - 1} & \text{ if $m = 3$ and $|Z(G)| \geq 2$; or  $m \geq 4$}\\
 \frac{(-m^2 + 2m)|Z(G)| + 2m - 1}{2m - 1} & \text{ if $m = 2$;  $m = 3$ and $|Z(G)| = 1$},

\end{cases}
$

$\left|2|Z(G)| - 2 - \frac{2|e(\Gamma_G)|}{|v(\Gamma_G)|}\right| = \begin{cases}
\frac{(5m - m^2 - 3)|Z(G)| - 2m + 1}{2m - 1} & \text{ if $m = 2$; $m = 3$ and $|Z(G)| \geq 2$};\\
& \text{ or $m = 4$ and $|Z(G)| > 6 $}\\
\frac{(-5m + m^2 + 3)|Z(G)| + 2m - 1}{2m - 1} & \text{ if  $m = 3$ and $|Z(G)| = 1$};\\
& \text{ $m = 4$ and $|Z(G)| \leq 6$; or $m \geq 5$}
\end{cases}
$ and

$\left||Z(G)| - 2 - \frac{2|e(\Gamma_G)|}{|v(\Gamma_G)|}\right| = \frac{(-3m + m^2 + 2)|Z(G)| + 2m - 1}{2m - 1}$.
\\
If $m = 2$, then by \eqref{Qenergy}, we have
\begin{align*}
LE^+&(\Gamma_G)\\
= &\frac{(3m^2 - 5m + 1)|Z(G)| - 2m + 1}{2m - 1} + \frac{((m - 1)|Z(G)| - 1)((-m^2 + 2m)|Z(G)| + 2m - 1)}{2m - 1}\\
& + \frac{m((5m - m^2 - 3)|Z(G)| - 2m + 1)}{2m - 1} + \frac{m(|Z(G)| - 1)((-3m + m^2 + 2)|Z(G)| + 2m - 1)}{2m - 1}\\
= &\frac{3|Z(G)| - 3}{3} + \frac{3(|Z(G)| - 1)}{3} + \frac{6|Z(G)| - 6}{3} + \frac{6(|Z(G)| - 1)}{3}\\
= & 6|Z(G)| - 6.
\end{align*}
\noindent
If $m = 3$ and $|Z(G)| = 1$, then by \eqref{Qenergy}, we have
\begin{align*}
LE^+&(\Gamma_G)\\
= &\frac{(3m^2 - 5m + 1)|Z(G)| - 2m + 1}{2m - 1} + \frac{((m - 1)|Z(G)| - 1)((-m^2 + 2m)|Z(G)| + 2m - 1)}{2m - 1}\\
& + \frac{m((-5m + m^2 + 3)|Z(G)| + 2m - 1)}{2m - 1} + \frac{m(|Z(G)| - 1)((-3m + m^2 + 2)|Z(G)| + 2m - 1)}{2m - 1}\\
= &\frac{8}{5} + \frac{2}{5} + \frac{6}{5}
= \frac{16}{5}.
\end{align*}
\noindent
If $m = 3$ and $|Z(G)| \geq 2$ then, by \eqref{Qenergy}, we have
\begin{align*}
LE^+&(\Gamma_G)\\
= &\frac{(3m^2 - 5m + 1)|Z(G)| - 2m + 1}{2m - 1} + \frac{((m - 1)|Z(G)| - 1)((m^2 - 2m)|Z(G)| - 2m + 1)}{2m - 1}\\
& + \frac{m((5m - m^2 - 3)|Z(G)| - 2m + 1)}{2m - 1} + \frac{m(|Z(G)| - 1)((-3m + m^2 + 2)|Z(G)| + 2m - 1)}{2m - 1}\\
= &\frac{13|Z(G)| - 5}{5} + \frac{(2|Z(G)| - 1)(3|Z(G)| - 5)}{5} + \frac{9|Z(G)| - 15}{5} + \frac{(3|Z(G)| - 3)( 2|Z(G)| + 5)}{5}\\
= &\frac{13|Z(G)| - 5}{5} + \frac{6|Z(G)|^2 - 13|Z(G)| + 5}{5} + \frac{9|Z(G)| - 15}{5} + \frac{6|Z(G)|^2 + 9|Z(G)| -15}{5}\\
= & \frac{12|Z(G)|^2 + 18|Z(G)| - 30}{5}.
\end{align*}
\noindent
If $m = 4$ and $|Z(G)| \leq 6$ then, by \eqref{Qenergy}, we have
\begin{align*}
LE^+&(\Gamma_G)\\
= &\frac{(3m^2 - 5m + 1)|Z(G)| - 2m + 1}{2m - 1} +  \frac{((m - 1)|Z(G)| - 1)((m^2 - 2m)|Z(G)| - 2m + 1)}{2m - 1}\\
& + \frac{m((-5m + m^2 + 3)|Z(G)| + 2m - 1)}{2m - 1} + \frac{m(|Z(G)| - 1)((-3m + m^2 + 2)|Z(G)| + 2m - 1)}{2m - 1}\\
= &\frac{29|Z(G)| - 7}{7} +  \frac{(3|Z(G)| - 1)(8|Z(G)| - 7)}{7} + \frac{4(-|Z(G)| + 7)}{7} + \frac{4(|Z(G)| - 1)(6|Z(G)| + 7)}{7}\\
%= &\frac{29|Z(G)| - 7}{7} + \frac{-24|Z(G)|^2  + 29|Z(G)| - 7}{7} + \frac{-4|Z(G)| + 28}{7} + \frac{24|Z(G)|^2 + 4|Z(G)| - 28}{7}\\
= & \frac{48|Z(G)| ^ 2}{7}.
\end{align*}
\noindent
If $m = 4$ and $|Z(G)| > 6$ then, by \eqref{Qenergy}, we have
\begin{align*}
LE^+&(\Gamma_G)\\
= &\frac{(3m^2 - 5m + 1)|Z(G)| - 2m + 1}{2m - 1} + \frac{((m - 1)|Z(G)| - 1)((m^2 - 2m)|Z(G)| - 2m + 1)}{2m - 1}\\
& + \frac{m((5m - m^2 - 3)|Z(G)| - 2m + 1)}{2m - 1} + \frac{m(|Z(G)| - 1)((-3m + m^2 + 2)|Z(G)| + 2m - 1)}{2m - 1}\\
%= &\frac{29|Z(G)| - 7}{7} + \frac{(3|Z(G)| - 1)(-8|Z(G)| + 7)}{7} + \frac{4(|Z(G)| - 7)}{7} + \frac{4(|Z(G)| - 1)(6|Z(G)| + 7)}{7}\\
= &\frac{29|Z(G)| - 7}{7} + \frac{24|Z(G)|^2  - 29|Z(G)| + 7}{7} + \frac{4|Z(G)| - 28}{7} + \frac{24|Z(G)|^2 + 4|Z(G)| - 28}{7}\\
= & \frac{48|Z(G)| ^ 2 + 8|Z(G)|- 56}{7}.
\end{align*}
\noindent
If $m \geq 5$  then, by \eqref{Qenergy}, we have
\begin{align*}
LE^+&(\Gamma_G)\\
= &\frac{(3m^2 - 5m + 1)|Z(G)| - 2m + 1}{2m - 1} + \frac{((m - 1)|Z(G)| - 1)((m^2 - 2m)|Z(G)| - 2m + 1)}{2m - 1}\\
& + \frac{m((-5m + m^2 + 3)|Z(G)| + 2m - 1)}{2m - 1} + \frac{m(|Z(G)| - 1)((-3m + m^2 + 2)|Z(G)| + 2m - 1)}{2m - 1}\\
%= &\frac{(3m^2 - 5m + 1)|Z(G)| - 2m + 1}{2m - 1} + \frac{(-m^3 + 3m^2 - 2m)|Z(G)|^2 + (3m^2 - 5m + 1)|Z(G)| - 2m + 1}{2m - 1}\\
%& + \frac{(-5m^2 + m^3 + 3m)|Z(G)| + 2m^2 - m}{2m - 1} + \frac{(-3m^2 + m^3 + 2m)|Z(G)|^2 + (m^2 + 2m - 2)|Z(G)| - 2m + 1}{2m - 1}\\
= & \frac{(2m^3 - 6m^2 + 4m)|Z(G)|^2 + (m^3 - 7m^2 + 4m)|Z(G)| - 2m^2 + 3m - 1}{2m - 1}.
\end{align*}
\end{proof}

Using Theorem \ref{main4}, we now compute the energy,  Laplacian energy and signless Laplacian energy of the commuting graphs of the groups $M_{2mn}, D_{2m}$ and $Q_{4n}$ respectively.

\begin{corollary}\label{main05}
Let $M_{2mn} = \langle a, b : a^m = b^{2n} = 1, bab^{-1} = a^{-1} \rangle$ be a metacyclic group, where $m > 2$.
\par
If $m$ is odd then,
\[
E(\Gamma_{M_{2mn}}) = (4m - 2)n - 2(m + 1),
\]

\[
LE(\Gamma_{M_{2mn}}) =
\begin{cases}
  \frac {56n - 40}{5}, & \mbox{if $m = 3$ and $n = 1,2$ }; \\
  \frac{12n^2 + 14n - 10}{5}, & \mbox{if $m = 3$ and $n \geq 3$};\\
  \frac{(2m^3 - 6m^2 + 4m)n^2 + (2m^2 - 2m + 2)n - 4m +2}{2m - 1} & \mbox{otherwise},
 \end{cases}
\]
and
\[
LE^+(\Gamma_{M_{2mn}}) =
\begin{cases}
  \frac {16}{5}, & \mbox{if $m = 3$ and $n = 1$ }; \\
  \frac{12n^2 + 18n - 30}{5}, & \mbox{if $m = 3$ and $n \geq 2$};\\
  \frac{(2m^3 - 6m^2 + 4m)n^2 + (m^3 - 7m^2 + 4m)n -2m^2 + 3m - 1}{2m - 1} & \mbox{otherwise}.
\end{cases}
\]

If $m$ is even then,
\[
E(\Gamma_{M_{2mn}}) = (4m - 4)n - (m + 2),
\]

\[
LE(\Gamma_{M_{2mn}}) =
\begin{cases}
  \frac {16n - 9}{3}, & \mbox{if $m = 4$}; \\
  \frac{72}{5}, & \mbox{if $m = 6$ and $n = 1$}; \\
  \frac{48n^2 + 28n - 10}{5}, & \mbox{if $m = 6$ and $n \geq 2$};\\
  \frac{192n^2 + 52n - 14}{7}, & \mbox{if $m = 8$ and $n \geq 1$ };\\
  \frac{(m^3 - 6m^2 + 8m)n^2 + (m^2 - 2m + 4)n - 2m +2}{m - 1}, & \mbox{otherwise };\\
 \end{cases}
\]

\[
LE^+(\Gamma_{M_{2mn}}) =
\begin{cases}
  12n - 6, & \mbox{if $m = 4$}; \\
  \frac{48n^2 + 36n - 30}{5}, & \mbox{if $m = 6$ and $n \geq 1$}; \\
  \frac{192n^2}{7}, & \mbox{if $m = 8$ and $n \leq 3$};\\
  \frac{192n^2 + 16n - 56}{7}, & \mbox{if $m = 8$ and $n > 3$ };\\
  \frac{(4m^3 - 24m^2 + 32m)n^2 + (m^3 - 14m^2 + 16m)n - 2m^2 + 6m - 4}{4(m - 1)}, & \mbox{otherwise }.\\
 \end{cases}
\]

\end{corollary}

\begin{proof}
Observe that $Z(M_{2mn}) = \langle b^2 \rangle$ or $\langle b^2 \rangle \cup a^{\frac{m}{2}}\langle b^2 \rangle$ according as $m$ is odd or even.  Also, it is easy to see that $\frac{M_{2mn}}{Z(M_{2mn})} \cong D_{2m}$ or $D_m$ according as $m$ is odd or even. Hence, the result follows from Theorem \ref{main4}.
\end{proof}
\noindent As a corollary to the above result we have the following results.

\begin{corollary}\label{main005}
Let $D_{2m} = \langle a, b : a^m = b^{2} = 1, bab^{-1} = a^{-1} \rangle$ be  the dihedral group of order $2m$, where $m > 2$. Then

If $m$ is odd, then
\[
E(\Gamma_{D_{2m}}) = 2m - 3,
\]

\[
LE(\Gamma_{D_{2m}}) =
\begin{cases}
  \frac {16}{5}, & \mbox{if $m = 3$ }; \\
  \frac{2(m + 1)(m - 1)(m - 2)}{2m - 1} & \mbox{otherwise},
 \end{cases} \quad \text{and} \quad
LE^+(\Gamma_{D_{2m}}) =
\begin{cases}
  \frac {16}{5}, & \mbox{if $m = 3$ }; \\
  \frac{3m^3 - 15m^2 + 11m - 1}{2m - 1} & \mbox{otherwise}.
 \end{cases}
\]

If $m$ is even, then
\[
E(\Gamma_{D_{2m}}) = 3m - 6,
\]

\[
LE(\Gamma_{D_{2m}}) =
\begin{cases}
  \frac {7}{3}, & \mbox{if $m = 4$ }; \\
  \frac {72}{5}, & \mbox{if $m = 6$ }; \\
  \frac {230}{7}, & \mbox{if $m = 8$ }; \\
  \frac{m^3 - 5m^2 + 4m + 6}{m - 1} & \mbox{otherwise},
 \end{cases}
 \quad \text{and} \quad
LE^+(\Gamma_{D_{2m}}) =
\begin{cases}
  6, & \mbox{if $m = 4$ }; \\
  \frac {54}{5}, & \mbox{if $m = 6$ }; \\
  \frac {192}{7}, & \mbox{if $m = 8$ }; \\
  \frac{5m^3 - 40m^2 + 42m - 4}{4(m - 1)} & \mbox{otherwise}.
 \end{cases}
\]

\end{corollary}

\begin{corollary}\label{q4m}
Let $Q_{4m} = \langle x, y : y^{2m} = 1, x^2 = y^m,yxy^{-1} = y^{-1}\rangle$, where $m \geq 2$, be the   generalized quaternion group of order $4m$. Then

\[
E(\Gamma_{Q_{4m}}) = 6m - 6,
\]

\[
LE(\Gamma_{Q_{4m}}) =
\begin{cases}
  6, & \mbox{if $m = 2$ }; \\
  \frac {72}{5}, & \mbox{if $m = 3$ }; \\
  \frac {230}{7}, & \mbox{if $m = 4$ }; \\
  \frac{8m^3 - 20m^2 + 8m + 6}{2m - 1} & \mbox{otherwise},
 \end{cases}
 \quad \text{and} \quad
LE^+(\Gamma_{Q_{4m}}) =
\begin{cases}
  6, & \mbox{if $m = 2$ }; \\
  \frac {54}{5}, & \mbox{if $m = 3$ }; \\
  \frac {192}{7}, & \mbox{if $m = 4$ }; \\
  \frac{10m^3 - 40m^2 + 27m - 1}{2m - 1} & \mbox{otherwise}.
 \end{cases}
\]
\end{corollary}
\begin{proof}
The result follows from Theorem \ref{main4} noting that  $Z(Q_{4m}) = \{1, a^m\}$ and  $\frac{Q_{4m}}{Z(Q_{4m})} \cong D_{2m}$.
\end{proof}

%\section{Some well-known groups}
In this section,
Now  we compute various energies of  the  commuting graphs of some  well-known families of finite non-abelian groups.

\begin{proposition}\label{order-pq}
Let $G$ be a non-abelian group of order $pq$, where $p$ and $q$ are primes with $p\mid (q - 1)$. Then
\[
E(\Gamma_G) = 2q(p - 1) - 3, \quad
LE(\Gamma_G) =
\begin{cases}
  \frac{q(q^2 -3q -3pq^2 + 1)}{pq - 1}, & \mbox{if $p=2$ and $q\neq 3$ }, \\
  \frac{2pq(2pq - p - q^2 - 3q + 1)+q(5q^2 - 6q + 4)}{pq - 1}, & \mbox{if $p=2$ and $q=3$  }, \\
  \frac{-2pq(pq - 2p -q^2 + 4) - q(3q^2 - 6q + 2) + 4}{pq - 1}, & \mbox{otherwise};
\end{cases}
\]
and
\[
LE^+(\Gamma_G) =
\begin{cases}
  \frac{2pq(2q - p - 1)-(2q^2 + 3q - 6)}{pq - 1}, & \mbox{if $p=2$ and $q = 3$ }, \\
  \frac{2p^2q(1 - q) + 2q^3(p - 1) + q(2q - 2p + 1) - 2}{pq - 1}, & \mbox{otherwise}.
\end{cases}
\]
\end{proposition}

\begin{proof}
By \cite[Lemma 3]{Dutta16}, we have
\[
\spec(\Gamma_G) = \{(-1)^{pq - q -1}, (p - 2)^q, (q - 2)^1\}.
\]
Therefore, by \eqref{energy} we have
\[
E(\Gamma_G) = 2q(p - 1) - 3.
\]
Note  that $|v(\Gamma_G)| = pq - 1$ and $|e(\Gamma_G)| = \frac{p^2q - 3pq + q^2 - q + 2}{2}$ since $\Gamma_G = qK_{p - 1} \sqcup  K_{q-1}$. Therefore,
\[
\frac{2|e(\Gamma_G)|}{|v(\Gamma_G)|} = \frac{p^2q - 3pq + q^2 - q + 2}{pq - 1}.
\]

By \cite[Proposition 2.9]{dn016}, we have
\[
\L-spec(\Gamma_G) =  \{0^{q + 1}, (q - 1)^{q - 2}, (p - 1)^{pq - 2q}\}.
\]
\noindent
Therefore,
\[
\left|0 - \frac{2|e(\Gamma_G)|}{|v(\Gamma_G)|}\right| = \frac{p^2q - 3pq + q^2 - q + 2}{pq - 1}, \quad
\left|q - 1 - \frac{2|e(\Gamma_G)|}{|v(\Gamma_G)|}\right| =
\begin{cases}
  \frac{-pq(q - p) + 2q(q - p) + 1}{pq - 1}, & \mbox{if $p=2$ }; \\
  \frac{pq(q - p) - 2q(q - p) - 1}{pq - 1}, & \mbox{otherwise}
\end{cases}
\] and

\[
\left|p - 1 - \frac{2|e(\Gamma_G)|}{|v(\Gamma_G)|}\right| =
\begin{cases}
  \frac{q(2p - q) + (q - p) - 1}{pq - 1}, & \mbox{if $p=2$ and $q=2$ }; \\
  \frac{- q(2p - q) - (q - p) + 1}{pq - 1}, & \mbox{otherwise}.
\end{cases}
\]
\noindent
Hence, by \eqref{Lenergy} we have, if $p=2$ and $q\neq 3$, then
\begin{align*}
LE(\Gamma_G) = & \frac{(q + 1)(p^2q - 3pq + q^2 - q + 2)}{pq-1} + \frac{(q - 2)(-pq(q - p) + 2q(q - p) + 1)}{pq - 1}\\
& + \frac{(pq - 2q)(- q(2p - q) - (q - p) + 1)}{pq - 1}\\
= & \frac{q(q^2 - 3q -3pq^2 + 1)}{pq - 1}.
\end{align*}
\noindent
If $p=2$ and $q = 3$, then
\begin{align*}
LE(\Gamma_G) = & \frac{(q + 1)(p^2q - 3pq + q^2 - q + 2)}{pq-1} + \frac{(q - 2)(-pq(q - p) + 2q(q - p) + 1)}{pq - 1}\\
& + \frac{(pq - 2q)(q(2p - q) + (q - p) - 1)}{pq - 1}\\
= & \frac{2pq(2pq - p - q^2 - 3q + 1) + q(5q^2 - 6q + 4)}{pq - 1}.
\end{align*}
\noindent
Otherwise,
\begin{align*}
LE(\Gamma_G) = & \frac{(q + 1)(p^2q - 3pq + q^2 - q + 2)}{pq-1} + \frac{(q - 2)(pq(q - p) - 2q(q - p) - 1)}{pq - 1}\\
& + \frac{(pq - 2q)(- q(2p - q) - (q - p) + 1)}{pq - 1}\\
= & \frac{-2pq(pq - 2p -q^2 + 4) - q(3q^2 - 6q + 2) + 4}{pq - 1}.
\end{align*}

By \cite[Proposition 2.9]{dn016}, we also have
\[
\Q-spec(\Gamma_G) =  \{(2q - 4)^1, (q - 3)^{q - 2}, (2p - 4)^q, (p - 3)^{pq - 2q}\}.
\]
\noindent
Therefore,
$\left|2q - 4 - \frac{2|e(\Gamma_G)|}{|v(\Gamma_G)|}\right| = \frac{pq(2q - p) - q(p + q + 1) + 2}{pq - 1}$,

\[
\left|q - 3 - \frac{2|e(\Gamma_G)|}{|v(\Gamma_G)|}\right| =
\begin{cases}
  \frac{-pq(q - p) + q^2 - 2}{pq - 1}, & \mbox{if $p=2$ and $q=3$ }; \\
  \frac{pq(q - p) - q^2 + 2}{pq - 1}, & \mbox{otherwise},
\end{cases}
\]

$\left|2p - 4 - \frac{2|e(\Gamma_G)|}{|v(\Gamma_G)|}\right| = \frac{-pq(p - 1) + q(q - 1) +2p - 2}{pq - 1}$ and $\left|p - 3 - \frac{2|e(\Gamma_G)|}{|v(\Gamma_G)|}\right| = \frac{p + q(q - 1) - 1}{pq - 1}$.
\noindent
Therefore, by \eqref{Qenergy} we have, if $p=2$ and $q=3$,

\begin{align*}
  LE^+(\Gamma_G) = & \frac{pq(2q - p) - q(p + q + 1) + 2}{pq - 1} + \frac{(q - 2)(-pq(q - p) + q^2 - 2)}{pq - 1} \\
   + & \frac{q(-pq(p - 1) + q(q - 1) +2p - 2)}{pq - 1} + \frac{(pq - 2q)(p + q(q - 1) - 1)}{pq - 1}\\
   = & \frac{2pq(2q - p - 1)-(2q^2 + 3q - 6)}{pq - 1}.
\end{align*}
\noindent
Otherwise,
\begin{align*}
  LE^+(\Gamma_G) = & \frac{pq(2q - p) - q(p + q + 1) + 2}{pq - 1} + \frac{(q - 2)(pq(q - p) - q^2 + 2)}{pq - 1} \\
   + & \frac{q(-pq(p - 1) + q(q - 1) +2p - 2)}{pq - 1} + \frac{(pq - 2q)(p + q(q - 1) - 1)}{pq - 1}\\
   = & \frac{2p^2q(1 - q) + 2q^3(p - 1) + q(2q - 2p + 1) - 2}{pq - 1}.
\end{align*}
This completes the proof.
\end{proof}

\begin{proposition}\label{semid}
Let $QD_{2^n}$ denote the quasidihedral group $\langle a, b : a^{2^{n-1}} =  b^2 = 1, bab^{-1} = a^{2^{n - 2} - 1}\rangle$, where $n \geq 4$. Then
\[
E(\Gamma_QD_{2^n}) =  3(2^{n - 1} - 2),
\]
\[
LE(\Gamma_QD_{2^n}) =  \frac{2^{3n - 3} - 5.2^{2n - 2} + 4.2^{n - 1} + 12}{2^{n - 1} - 1}
\] and
\[
LE^+(\Gamma_QD_{2^n}) =  \frac{5.2^{3n - 4} - 30.2^{2n - 3} + 40.2^{n - 2}}{2^{n - 1} - 1}.
\]
\end{proposition}
\begin{proof}
By \cite[Proposition 1]{Dutta16}, we have
\[
\spec(\Gamma_{QD_{2^n}}) = \{(-1)^{2^{n}  - 2^{n - 2} - 3}, 1^{2^{n - 2}}, (2^{n - 1} - 3)^1\}.
\]
Therefore, by \eqref{energy} we have
\[
E(\Gamma_{QD_{2^n}}) = 3(2^{n - 1} - 2).
\]
Note  that $|v(\Gamma_{QD_{2^n}})| = 2^{n - 1} - 1$ and $|e(\Gamma_{QD_{2^n}})| = \frac{2^{2n - 2} - 4.2^{n - 1} + 6}{2}$ since $\Gamma_G =2^{n - 2} K_2 \sqcup  K_{2^{n - 1} - 2} $. Therefore,
\[
\frac{2|e(\Gamma_{QD_{2^n}})|}{|v(\Gamma_{QD_{2^n}})|} = \frac{2^{2n - 2} - 4.2^{n - 1} + 6}{2(2^{n - 1} - 1)}.
\]

By \cite[Proposition 2.10]{dn016}, we have
\[
\L-spec(\Gamma_{QD_{2^n}}) =  \{0^{2^{n - 2} + 1}, (2^{n - 1} - 2)^{2^{n - 1}-3}, 2^{2^{n - 2}}\}.
\]
Therefore, $\left|0 - \frac{2|e(\Gamma_{QD_{2^n}})|}{|v(\Gamma_{QD_{2^n}})|}\right| = \frac{2^{2n - 2} - 4.2^{n - 1} + 6}{2(2^{n - 1} - 1)}$, $\left|2^{n - 1} - 2 - \frac{2|e(\Gamma_{QD_{2^n}})|}{|v(\Gamma_G)|}\right| = \frac{2^{2n - 2} - 2.2^{n - 1} - 2}{2(2^{n - 1} - 1)}$ and  \\ $\left|2 - \frac{2|e(\Gamma_{QD_{2^n}})|}{|v(\Gamma_G)|}\right| = \frac{2^{2n - 2} - 8.2^{n - 1} + 10}{2(2^{n - 1} - 1)}$.
Hence, by \eqref{Lenergy} we have
\[
LE(\Gamma_{QD_{2^n}}) = \frac{2^{3n - 3} - 5.2^{2n - 2} + 4.2^{n - 1} + 12}{2^{n - 1} - 1}.
\]

By \cite[Proposition 2.10]{dn016}, we also have
\[
\Q-spec(\Gamma_{QD_{2^n}}) =  \{(2^n - 6)^1, (2^{n - 1} - 4)^{{2^{n - 1} - 3}}, 2^{2^{n - 2}}, 0^{2^{n - 2}}\}.
\]
Therefore, $\left|2^n - 6 - \frac{2|e(\Gamma_{QD_{2^n}})|}{|v(\Gamma_{QD_{2^n}})|}\right| = \frac{3.2^{2n - 2} - 12.2^{n - 1} + 6}{2(2^{n - 1} - 1)}$, $\left|2^{n - 1} - 4 - \frac{2|e(\Gamma_{QD_{2^n}})|}{|v(\Gamma_{QD_{2^n}})|}\right| = \frac{2^{2n - 2} - 6.2^{n - 1} + 2}{2(2^{n - 1} - 1)}$,\\ $\left|2  - \frac{2|e(\Gamma_{QD_{2^n}})|}{|v(\Gamma_{QD_{2^n}})|}\right| = \frac{2^{2n - 2} - 8.2^{n - 1} + 10}{2(2^{n - 1} - 1)}$ and $\left|0 - \frac{2|e(\Gamma_{QD_{2^n}})|}{|v(\Gamma_{QD_{2^n}})|}\right| = \frac{2^{2n - 2} - 4.2^{n - 1} + 6}{2(2^{n - 1} - 1)}$.
Therefore, by \eqref{Qenergy} we have
\begin{align*}
  LE^+(\Gamma_{QD_{2^n}}) = & \frac{3.2^{2n - 2} - 12.2^{n - 1} + 6}{2(2^{n - 1} - 1)} + \frac{(2^{n - 1} - 3)(2^{2n - 2} - 6.2^{n - 1} + 2}{2(2^{n - 1} - 1)}  \\
  + & + \frac{2^{n - 2}(2^{2n - 2} - 8.2^{n - 1} + 10}{2(2^{n - 1} - 1)} + + \frac{2^{n - 2}(2^{2n - 2} - 4.2^{n - 1} + 6}{2(2^{n - 1} - 1)}\\
  & =  \frac{5.2^{3n - 4} - 30.2^{2n - 3} + 40.2^{n - 2}}{2^{n - 1} - 1}.
\end{align*}
\end{proof}

\begin{proposition}\label{psl}
Let $G$ denote the projective special linear group  $PSL(2, 2^k)$, where $k \geq 2$. Then
\[
E(\Gamma_G) =  2^{3k + 1} - 2^{2k + 1} - 2^{k + 2} - 4,
\]

\[
LE(\Gamma_G) =\frac{2.2^{6k} - 2.2^{5k} - 8.2^{4k} - 6.2^{3k} + 6.2^{2k} + 8.2^k + 4}{2^{3k} - 2^k - 1},
\]
and

\[
LE^+(\Gamma_G) =\begin{cases}
  \frac{2^{6k} + 2^{5k} - 3.2^{4k} - 7.2^{3k} + 4.2^k + 4}{2^{3k} - 2^k - 1}, & \mbox{if k=2}; \\
  \frac{2.2^{6k} - 2.2^{5k} - 8.2^{4k} - 6.2^{3k} + 6.2^{2k} + 8.2^k + 4}{2^{3k} - 2^k - 1}, & \mbox{otherwise}.
\end{cases}
\]
\end{proposition}

\begin{proof}
By \cite[Proposition 2]{Dutta16}, we have
\[
\spec(\Gamma_G) =  \{(-1)^{2^{3k} - 2^{2k} - 2^{k + 1} - 2}, (2^k - 1)^{2^{k - 1}(2^k - 1)}, (2^k - 2)^{2^k + 1}, (2^k - 3)^{2^{k - 1}(2^k + 1)}\}.
\]
Therefore, by \eqref{energy} we have
\[
E(\Gamma_G) = 2^{3k + 1} - 2^{2k + 1} - 2^{k + 2} - 4.
\]
Note  that $|v(\Gamma_G)| = 2^{3k} - 2^k - 1$ and $|e(\Gamma_G)| = \frac{2^{4k} - 2.2^{3k} - 2^{2k} + 2.2^k + 2}{2}$ since $\Gamma_G =  (2^k + 1)K_{2^k - 1}$ $\sqcup 2^{k - 1}(2^k + 1)K_{2^k - 2} \sqcup 2^{k - 1}(2^k - 1)K_{2^k}$. Therefore,
\[
\frac{2|e(\Gamma_G)|}{|v(\Gamma_G)|} = \frac{2^{4k} - 2.2^{3k} - 2^{2k} + 2.2^k + 2}{2^{3k} - 2^k - 1}.
\]

By \cite[Proposition 2.11]{dn016}, we have
\[
\L-spec(\Gamma_G) =   \{0^{2^{2k} + 2^k + 1}, (2^k - 1)^{2^{2k} - 2^k - 2}, (2^k - 2)^{2^{k - 1}(2^{2k} - 2^{k + 1} - 3)}, (2^k)^{2^{k - 1}(2^{2k} - 2^{k + 1} + 1)}\}.
\]
Therefore, $\left|0  - \frac{2|e(\Gamma_G)|}{|v(\Gamma_G)|}\right| = \frac{2^{4k} - 2.2^{3k} - 2^{2k} + 2.2^k + 2}{2^{3k} - 2^k - 1}$, $\left|2^k - 1  - \frac{2|e(\Gamma_G)|}{|v(\Gamma_G)|}\right| = \frac{2^{3k} - 2.2^k - 1}{2^{3k} - 2^k - 1}$, $\left|2^k - 2  - \frac{2|e(\Gamma_G)|}{|v(\Gamma_G)|}\right|  =$\\  $\frac{2^k}{2^{3k} - 2^k - 1}$ and $\left|2^k  - \frac{2|e(\Gamma_G)|}{|v(\Gamma_G)|}\right| = \frac{2.2^{3k} - 3.2^k - 2}{2^{3k} - 2^k - 1}$.
Hence, by \eqref{Lenergy} we have
\begin{align*}
  LE(\Gamma_G)  = & \frac{(2^{2k} + 2^k + 1)(2^{4k} - 2.2^{3k} - 2^{2k} + 2.2^k + 2)}{2^{3k} - 2^k - 1} + \frac{(2^{2k} - 2^k - 2)(2^{3k} - 2.2^k - 1)}{2^{3k} - 2^k - 1} \\
   & + \frac{2^{k - 1}(2^{2k} - 2^{k + 1} - 3)2^k}{2^{3k} - 2^k - 1} +  \frac{2^{k - 1}(2^{2k} - 2^{k + 1} + 1)(2.2^{3k} - 3.2^k - 2)}{2^{3k} - 2^k - 1}  \\
  = & \frac{2.2^{6k} - 2.2^{5k} - 3.2^{4k} - 4.2^{3k} + 3.2^{2k} + 8.2^k + 4}{2^{3k} - 2^k - 1}.
\end{align*}

By \cite[Proposition 2.11]{dn016}, we also have
\begin{align*}
\Q-spec(\Gamma_G) = \{(2^{k + 1} - 4)^{2^k + 1}, (2^k - 3)^{2^{2k} - 2^k - 2}, & (2^{k + 1} - 6)^{2^{k - 1}(2^k + 1)}, (2^k - 4)^{2^{k - 1}(2^{2k} - 2^{k + 1} - 3)},\\
& (2^{k + 1} - 2)^{2^{k - 1}(2^k - 1)}, (2^k - 2)^{2^{k - 1}(2^{2k} - 2^{k + 1} + 1)}\}.
\end{align*}
Therefore, $\left|2^{k + 1} - 4  - \frac{2|e(\Gamma_G)|}{|v(\Gamma_G)|}\right| =  \frac{2^{4k} - 2.2^{3k} - 2^{2k} + 2}{2^{3k} - 2^k - 1}$, $\left|2^k - 3  - \frac{2|e(\Gamma_G)|}{|v(\Gamma_G)|}\right| =  \frac{2^{3k} - 1}{2^{3k} - 2^k - 1}$,
\[
\left|2^{k + 1} - 6   - \frac{2|e(\Gamma_G)|}{|v(\Gamma_G)|}\right| =
\begin{cases}
  \frac{- 2^{4k} + 4.2^{3k} + 2^{2k} - 2.2^k - 4}{2^{3k} - 2^k - 1}, & \mbox{if k=2 };\\
  \frac{ 2^{4k} - 4.2^{3k} - 2^{2k} + 2.2^k + 4}{2^{3k} - 2^k - 1}, & \mbox{otherwise},
\end{cases}
\]

  $\left|2^k - 4  - \frac{2|e(\Gamma_G)|}{|v(\Gamma_G)|}\right| = \frac{2.2^{3k} - 2^k - 2}{2^{3k} - 2^k - 1}$, $\left|2^{k + 1} - 2  - \frac{2|e(\Gamma_G)|}{|v(\Gamma_G)|}\right| = \frac{2^{4k} - 2^{2k} - 2.2^k}{2^{3k} - 2^k - 1}$ and $\left|2^k - 2  - \frac{2|e(\Gamma_G)|}{|v(\Gamma_G)|}\right| = \frac{2^k}{2^{3k} - 2^k - 1}$.
Therefore, by \eqref{Qenergy} we have, if $k=2$ then

\begin{align*}
  LE^+(\Gamma_G) = & \frac{(2^k + 1)(2^{4k} - 2.2^{3k} - 2^{2k} + 2)}{2^{3k} - 2^k - 1} +  \frac{(2^{2k} - 2^k - 2)(2^{3k} - 1)}{2^{3k} - 2^k - 1}\\
   & + \frac{2^{k - 1}(2^k + 1)(- 2^{4k} + 4.2^{3k} + 2^{2k} - 2.2^k - 4)}{2^{3k} - 2^k - 1} + \frac{2^{k - 1}(2^{2k} - 2^{k + 1} - 3)(2.2^{3k} - 2^k - 2)}{2^{3k} - 2^k - 1} \\
   & + \frac{2^{k - 1}(2^k - 1)(2^{4k} - 2^{2k} - 2.2^k)}{2^{3k} - 2^k - 1} +  \frac{2^{k - 1}(2^{2k} - 2^{k + 1} + 1)2^k}{2^{3k} - 2^k - 1}\\
  = & \frac{2^{6k} + 2^{5k} - 3.2^{4k} - 7.2^{3k} + 4.2^k + 4}{2^{3k} - 2^k - 1}.
\end{align*}
\noindent
Otherwise,
\begin{align*}
  LE^+(\Gamma_G) = & \frac{(2^k + 1)(2^{4k} - 2.2^{3k} - 2^{2k} + 2)}{2^{3k} - 2^k - 1} +  \frac{(2^{2k} - 2^k - 2)(2^{3k} - 1)}{2^{3k} - 2^k - 1}\\
   & + \frac{2^{k - 1}(2^k + 1)(2^{4k} - 4.2^{3k} - 2^{2k} + 2.2^k + 4)}{2^{3k} - 2^k - 1} + \frac{2^{k - 1}(2^{2k} - 2^{k + 1} - 3)(2.2^{3k} - 2^k - 2)}{2^{3k} - 2^k - 1} \\
   & + \frac{2^{k - 1}(2^k - 1)(2^{4k} - 2^{2k} - 2.2^k)}{2^{3k} - 2^k - 1} +  \frac{2^{k - 1}(2^{2k} - 2^{k + 1} + 1)2^k}{2^{3k} - 2^k - 1}\\
  = & \frac{2.2^{6k} - 2.2^{5k} - 8.2^{4k} - 6.2^{3k} + 6.2^{2k} + 8.2^k + 4}{2^{3k} - 2^k - 1}.
\end{align*}
This completes the proof.
\end{proof}

\begin{proposition}
Let $G$ denote the  general linear group  $GL(2, q)$, where $q = p^n > 2$ and $p$ is a prime. Then
\begin{align*}
E(\Gamma_G) = & \frac{2q^4 - 2q^3 - 8q^2 - 5q}{2},\\
LE(\Gamma_G) = & \frac{2q^9 - 6q^8 + 4q^7 + 8q^6 -10q^5 + 4q^3 + 4q^2 -8q}{2(q-1)(q^3 - q - 1)}  \quad \text{ and } \\
LE^+(\Gamma_G) = & \frac{q^{10} - 4q^9 + 10q^8 + 3q^7 - 23q^6 -9q^5 + 22q^4 + 10q^3 - 9q^2 -4q}{2(q-1)(q^3 - q - 1)}.
\end{align*}
\end{proposition}

\begin{proof}
By \cite[Proposition 3]{Dutta16}, we have
\[
\spec(\Gamma_G) =  \{(-1)^{q^4 -q^3 - 2q^2 - q}, (q^2 -3q + 1)^{\frac{q(q + 1)}{2}}, (q^2 - q - 1)^{\frac{q(q - 1)}{2}}, (q^2 - 2q)^{q + 1}\}.
\]
Therefore, by \eqref{energy} we have
\[
E(\Gamma_G) = \frac{2q^4 - 2q^3 - 8q^2 - 5q}{2}.
\]
Note  that $|v(\Gamma_G)| = (q - 1)(q^3 - q -1)$ and $|e(\Gamma_G)| = \frac{q^6 - 3q^5 + q^4 + 3q^3 - q^2 - q}{2}$ as $\Gamma_G =  \frac{q(q + 1)}{2}K_{q^2 - 3q + 2}$  $\sqcup \frac{q(q - 1)}{2}K_{q^2 - q} \sqcup (q + 1)K_{q^2 - 2q + 1}$. Therefore,
\[
\frac{2|e(\Gamma_G)|}{|v(\Gamma_G)|} = \frac{q^6 - 3q^5 + q^4 + 3q^3 - q^2 - q}{(q - 1)(q^3 - q -1)}.
\]

By \cite[Proposition 2.12]{dn016}, we have
\[
\L-spec(\Gamma_G) =   \{0^{q^2 + q + 1}, (q^2 - 3q + 2)^{\frac{q(q + 1)(q^2 - 3q + 1)}{2}}, (q^2 - q)^{\frac{q(q - 1)(q^2 - q - 1)}{2}}, (q^2 - 2q + 1)^{q(q + 1)(q - 2)}\}.
\]
Therefore, $\left|0  - \frac{2|e(\Gamma_G)|}{|v(\Gamma_G)|}\right| = \frac{q^6 - 3q^5 + q^4 + 3q^3 - q^2 - q}{(q - 1)(q^3 - q -1)}$, $\left|q^2 - 3q + 2  - \frac{2|e(\Gamma_G)|}{|v(\Gamma_G)|}\right| = \frac{q^5 - 3q^4 + 2q^3 + 2q - 2}{(q - 1)(q^3 - q -1)}$,\\ $\left|q^2 - q - \frac{2|e(\Gamma_G)|}{|v(\Gamma_G)|}\right| = \frac{q^5 - q^4 - 2q^3 + 2q^2}{(q - 1)(q^3 - q -1)}$ and $\left|q^2 - 2q + 1  - \frac{2|e(\Gamma_G)|}{|v(\Gamma_G)|}\right| = \frac{q^4 - 2q^3 + q^2 - q + 1}{(q - 1)(q^3 - q -1)}$.

\noindent
Hence, by \eqref{Lenergy} we have
\begin{align*}
  LE(\Gamma_G)  = & \frac{(q^2 + q + 1)(q^6 - 3q^5 + q^4 + 3q^3 - q^2 - q)}{(q - 1)(q^3 - q -1)} +  \left(\frac{q(q + 1)(q^2 - 3q + 1)}{2}\right)\frac{q^5 - 3q^4 + 2q^3 + 2q - 2}{(q - 1)(q^3 - q -1)}\\
   & + \left(\frac{q(q - 1)(q^2 - q - 1)}{2}\right)\frac{q^5 - q^4 - 2q^3 + 2q^2}{(q - 1)(q^3 - q -1)} + \frac{q(q + 1)(q - 2)(q^4 - 2q^3 + q^2 - q + 1)}{(q - 1)(q^3 - q -1)}\\
   = & \frac{2q^9 - 6q^8 + 4q^7 + 8q^6 -10q^5 + 4q^3 + 4q^2 -8q}{2(q-1)(q^3 - q - 1)}.
\end{align*}

By \cite[Proposition 2.12]{dn016}, we also have
\begin{align*}
\Q-spec(\Gamma_G) = &\{(2q^2 - 6q - 2)^{\frac{q(q + 1)}{2}}, (q^2 - 3q)^{\frac{q(q + 1)(q^2 - 3q + 1)}{2}}, (2q^2 - 2q - 2)^{\frac{q(q - 1)}{2}},\\
& (q^2 - q - 2)^{\frac{q(q - 1)(q^2 - q - 1)}{2}},
 (2q^2 - 4q)^{q + 1}, (q^2 + 2q -1)^{q(q + 1)(q - 2)}\}.
\end{align*}
Therefore, $\left|2q^2 - 6q - 2  - \frac{2|e(\Gamma_G)|}{|v(\Gamma_G)|}\right| = \frac{q^6 + 7q^5 - 11q^4 - 7q^3 + 5q^2 + 7q - 2}{(q - 1)(q^3 - q -1)}$, $\left|q^2 - 3q  - \frac{2|e(\Gamma_G)|}{|v(\Gamma_G)|}\right| = \frac{q^5 - q^4 + q^2 + 2q}{(q - 1)(q^3 - q -1)}$,\\ $\left|2q^2 - 2q - 2   - \frac{2|e(\Gamma_G)|}{|v(\Gamma_G)|}\right| = \frac{q^6 + 3q^5 - 3q^4 + q^3 + 6q^2 - q - 2}{(q - 1)(q^3 - q -1)}$, $\left|q^2 - q - 2  - \frac{2|e(\Gamma_G)|}{|v(\Gamma_G)|}\right|$ $= \frac{q^5 - 3q^4 + 4q^2 - 2}{(q - 1)(q^3 - q -1)}$,
\\$\left|2q^2 - 4q   - \frac{2|e(\Gamma_G)|}{|v(\Gamma_G)|}\right| = \frac{q^6 - 3q^5 + q^4 + q^3 + 3q^2 - 3q}{(q - 1)(q^3 - q -1)}$ and $\left|q^2 + 2q -1  - \frac{2|e(\Gamma_G)|}{|v(\Gamma_G)|}\right| = \frac{4q^5 - 5q^4 - 4q^3 + 3q^2 + 3q - 1}{(q - 1)(q^3 - q -1)}$.
Therefore, by \eqref{Qenergy} we have
\begin{align*}
  LE^+(\Gamma_G) = & \left(\frac{q(q + 1)}{2}\right)\frac{q^6 + 7q^5 - 11q^4 - 7q^3 + 5q^2 + 7q - 2}{(q - 1)(q^3 - q -1)} + \left(\frac{q(q + 1)(q^2 - 3q + 1)}{2}\right)\frac{q^5 - q^4 + q^2 + 2q}{(q - 1)(q^3 - q -1)} \\
  & + \left(\frac{q(q - 1)}{2}\right)\frac{q^6 + 3q^5 - 3q^4 + q^3 + 6q^2 - q - 2}{(q - 1)(q^3 - q -1)} + \left(\frac{q(q - 1)(q^2 - q - 1)}{2}\right)\frac{q^5 - 3q^4 + 4q^2 - 2}{(q - 1)(q^3 - q -1)} \\
  & + \frac{(q + 1)(q^6 - 3q^5 + q^4 + q^3 + 3q^2 - 3q)}{(q - 1)(q^3 - q -1)} + \frac{q(q + 1)(q - 2)(4q^5 - 5q^4 - 4q^3 + 3q^2 + 3q - 1)}{(q - 1)(q^3 - q -1)} \\
  = & \frac{q^{10} - 4q^9 + 10q^8 + 3q^7 - 23q^6 -9q^5 + 22q^4 + 10q^3 - 9q^2 -4q}{2(q-1)(q^3 - q - 1)}.
\end{align*}
\end{proof}

\begin{proposition}\label{Hanaki1}
Let $F = GF(2^n), n \geq 2$ and $\vartheta$ be the Frobenius  automorphism of $F$, i. e., $\vartheta(x) = x^2$ for all $x \in F$. If $G$ denotes the group
\[
 \left\lbrace U(a, b) = \begin{bmatrix}
        1 & 0 & 0\\
        a & 1 & 0\\
        b & \vartheta(a) & 1
       \end{bmatrix} : a, b \in F \right\rbrace
\]
under matrix multiplication given by $U(a, b)U(a', b') = U(a + a', b + b' + a'\vartheta(a))$, then
\[
E(\Gamma_G) = LE(\Gamma_G) =  LE^+(\Gamma_G) = 2{(2^n - 1)}^2.
\]
\end{proposition}

\begin{proof}
By \cite[Proposition 4]{Dutta16}, we have
\[
\spec(\Gamma_G) =  \{(-1)^{(2^n - 1)^2}, (2^n - 1)^{2^n - 1}\}.
\]
Therefore, by \eqref{energy} we have
\[
E(\Gamma_G) = 2{(2^n - 1)}^2.
\]
Note  that $|v(\Gamma_G)| = 2^n(2^n - 1)$ and $|e(\Gamma_G)| = \frac{2^{3n} - 2^{2n + 1} + 2^n}{2}$ since $\Gamma_G =  (2^n - 1)K_{2^n}$. Therefore,
\[
\frac{2|e(\Gamma_G)|}{|v(\Gamma_G)|} = 2^n - 1.
\]

By \cite[Proposition 2.13]{dn016}, we have
\[
\L-spec(\Gamma_G) =   \{0^{2^n - 1}, (2^n)^{2^{2n} - 2^{n + 1} + 1}\}.
\]
Therefore, $\left|0  - \frac{2|e(\Gamma_G)|}{|v(\Gamma_G)|}\right| = 2^n - 1$ and $\left|2^n  - \frac{2|e(\Gamma_G)|}{|v(\Gamma_G)|}\right| = 1$.
Hence, by \eqref{Lenergy} we have

\begin{align*}
  LE(\Gamma_G) = & (2^n - 1)(2^n - 1) + (2^{2n} - 2^{n + 1} + 1)1 \\
  = & 2{(2^n - 1)}^2.
\end{align*}

By \cite[Proposition 2.13]{dn016}, we also have
\[
\Q-spec(\Gamma_G) =   \{(2^{n + 1} - 2)^{2^n - 1}, (2^n - 2)^{2^{2n} - 2^{n + 1} + 1}\}.
\]
Therefore, $\left|2^{n + 1} - 2  - \frac{2|e(\Gamma_G)|}{|v(\Gamma_G)|}\right| = 2^n - 1$ and $\left|2^n - 2  - \frac{2|e(\Gamma_G)|}{|v(\Gamma_G)|}\right| = 1$.
Therefore, by \eqref{Qenergy} we have
\begin{align*}
  LE^+(\Gamma_G) = & (2^n - 1)(2^n - 1) + (2^{2n} - 2^{n + 1} + 1)1 \\
  = & 2{(2^n - 1)}^2.
\end{align*}
\end{proof}

\begin{proposition}\label{Hanaki2}
Let $F = GF(p^n)$ where $p$ is a prime. If $G$ denotes    the  group
\[
\left\lbrace V(a, b, c) = \begin{bmatrix}
        1 & 0 & 0\\
        a & 1 & 0\\
        b & c & 1
       \end{bmatrix} : a, b, c \in F \right\rbrace
\]
under matrix multiplication $V(a, b, c)V(a', b', c') = V(a + a', b + b' + ca', c + c')$,
then
\[
E(\Gamma_G) = LE(\Gamma_G) = LE^+(\Gamma_G) = 2(p^{3n} - 2p^n - 1).
\]

\end{proposition}

\begin{proof}
By \cite[Proposition 5]{Dutta16}, we have
\[
\spec(\Gamma_G) =  \{(-1)^{p^{3n} -2p^{n} -  1},   (p^{2n} - p^n - 1)^{p^n + 1}\}.
\]
Therefore, by \eqref{energy} we have
\[
E(\Gamma_G) = 2(p^{3n} - 2p^n - 1).
\]
Note  that $|v(\Gamma_G)| = p^n(p^{2n} - 1)$ and $|e(\Gamma_G)| = \frac{p^{5n} - p^{4n} - 2p^{3n} + p^{2n} + p^n}{2}$ since $\Gamma_G =  (p^n + 1)K_{p^{2n} - p^n}$. Therefore,
\[
\frac{2|e(\Gamma_G)|}{|v(\Gamma_G)|} = p^{2n} - p^n - 1.
\]

By \cite[Proposition 2.14]{dn016}, we have
\[
\L-spec(\Gamma_G) =   \{0^{p^n + 1}, (p^{2n} - p^n)^{p^{3n} -2p^{n} -  1}\}.
\]
Therefore, $\left|0  - \frac{2|e(\Gamma_G)|}{|v(\Gamma_G)|}\right| = p^{2n} - p^n - 1$ and  $\left|p^{2n} - p^n  - \frac{2|e(\Gamma_G)|}{|v(\Gamma_G)|}\right| = 1$.
Hence, by \eqref{Lenergy} we have
\begin{align*}
  LE(\Gamma_G) & = (p^n + 1)(p^{2n} - p^n - 1) + (p^{3n} -2p^{n} -  1)1\\
  & =  2(p^{3n} - 2p^n - 1).
\end{align*}

By \cite[Proposition 2.14]{dn016}, we also have
\[
\Q-spec(\Gamma_G) =   \{(2p^{2n} - 2p^n - 2)^{p^n + 1}, (p^{2n} - p^n - 2)^{p^{3n} -2p^{n} -  1}\}.
\]
Therefore, $\left|2p^{2n} - 2p^n - 2  - \frac{2|e(\Gamma_G)|}{|v(\Gamma_G)|}\right| = p^{2n} - p^n - 1$  and $\left|p^{2n} - p^n - 2  - \frac{2|e(\Gamma_G)|}{|v(\Gamma_G)|}\right| = 1$.
Therefore, by \eqref{Qenergy} we have

\begin{align*}
  LE^+(\Gamma_G) & = (p^n + 1)(p^{2n} - p^n - 1) + (p^{3n} -2p^{n} -  1)1\\
  & =  2(p^{3n} - 2p^n - 1).
\end{align*}
\end{proof}

\section{Some consequences}

For a finite group $G$, the set $C_G(x) = \{y \in G : xy = yx\}$ is called the centralizer of an element $x \in G$. Let $|\cent(G)| = |\{C_G(x) : x \in G\}|$, that is the number of distinct centralizers in $G$. A group $G$ is called an $n$-centralizer group if $|\cent(G)| = n$. The study of these groups was initiated by  Belcastro and  Sherman   \cite{bG94} in the year 1994. The readers may conf. \cite{Dutta10} for various results on these groups. In this section, we compute various energies of the commuting graphs of non-abelian  $n$-centralizer finite groups for some positive integer $n$.   We begin with the following result.

\begin{theorem}\label{4-cent}
If $G$ is a finite $4$-centralizer group, then
\[
E(\Gamma_G) =  LE(\Gamma_G) = LE^+(\Gamma_G) = 6|Z(G)| - 6.
\]

\end{theorem}
\begin{proof}
Let $G$ be a finite $4$-centralizer group. Then, by  \cite[Theorem 2]{bG94}, we have  $\frac{G}{Z(G)} \cong {\mathbb{Z}}_2 \times {\mathbb{Z}}_2$. Therefore, by Theorem \ref{main2}, the result follows.
\end{proof}

\noindent Further, we have the following result.

\begin{corollary}
If $G$ is a  finite $(p+2)$-centralizer $p$-group for any prime $p$,  then
\[
E(\Gamma_G) =  LE(\Gamma_G) = LE^+(\Gamma_G) = 2(p^2 - 1)|Z(G)| - 2(p + 1).
\]
\end{corollary}
\begin{proof}
Let $G$ be a finite $(p + 2)$-centralizer $p$-group. Then, by   \cite[Lemma 2.7]{ali00}, we have  $\frac{G}{Z(G)} \cong {\mathbb{Z}}_p \times {\mathbb{Z}}_p$. Therefore, by Theorem \ref{main2}, the result follows.
\end{proof}

\begin{theorem}\label{5-cent}
If $G$ is a  finite $5$-centralizer  group, then

\[
E(\Gamma_G) =  LE(\Gamma_G) = LE^+(\Gamma_G) = 16|Z(G)| - 8.
\]

or
\[
E(\Gamma_G) = 10|Z(G)| - 8, \quad
LE(\Gamma_G) =
\begin{cases}
  \frac{56|Z(G)| - 40}{5}, & \mbox{if m = 3 and Z(G) = 1,2}; \\
  \frac{12|Z(G)|^2 + 11|Z(G)| - 10}{5}, & \mbox{otherwise}
\end{cases}
\]
and
\[
LE^+(\Gamma_G) =
\begin{cases}
  \frac{16}{5}, & \mbox{if m = 3 and Z(G) = 1}; \\
  \frac{12|Z(G)|^2 + 18|Z(G)| - 30}{5}, & \mbox{otherwise}.
\end{cases}
\]

\end{theorem}

\begin{proof}
Let $G$ be a finite $5$-centralizer group. Then by  \cite[Theorem 4]{bG94}, we have  $\frac{G}{Z(G)} \cong {\mathbb{Z}}_3 \times {\mathbb{Z}}_3$ or $D_6$. Now, if $\frac{G}{Z(G)} \cong {\mathbb{Z}}_3 \times {\mathbb{Z}}_3$, then  by Theorem \ref{main2}, we have

\[
E(\Gamma_G) =  LE(\Gamma_G) = LE^+(\Gamma_G) = 16|Z(G)| - 8.
\]

If $\frac{G}{Z(G)} \cong D_6$, then by Theorem \ref{main4} we have
\[
E(\Gamma_G) = 10|Z(G)| - 8,
\quad
LE(\Gamma_G) =
\begin{cases}
  \frac{56|Z(G)| - 40}{5}, & \mbox{if m = 3 and Z(G) = 1,2}; \\
  \frac{12|Z(G)|^2 + 11|Z(G)| - 10}{5}, & \mbox{otherwise}
\end{cases}
\]
and
\[
LE^+(\Gamma_G) =
\begin{cases}
  \frac{16}{5}, & \mbox{if m = 3 and Z(G) = 1}; \\
  \frac{12|Z(G)|^2 + 18|Z(G)| - 30}{5}, & \mbox{otherwise}.
\end{cases}
\]
\noindent This completes the proof.
\end{proof}

%\noindent We conclude this section with the  following corollary.
%\begin{corollary}
%Let $G$ be a finite non-abelian group and $\{x_1, x_2, \dots, x_r\}$ be a set of pairwise non-commuting elements of $G$ having maximal size. Then $G$ is super integral if $r = 3, 4$.
%\end{corollary}

%\begin{proof}
%By Lemma 2.4 in \cite{ajH07}, we have that $G$ is a $4$-centralizer or a $5$-centralizer group according as  $r = 3$ or $4$. Hence the result follows from Theorem \ref{4-cent} and Theorem \ref{5-cent}.
%\end{proof}
%\section{Commutativity degree and super integral group}
Let $G$ be a finite group. The commutativity degree of $G$ is given by the ratio
\[
\Pr(G) = \frac{|\{(x, y) \in G \times G : xy = yx\}|}{|G|^2}.
\]
The origin of the commutativity degree of a finite group lies in a paper of Erd$\ddot{\rm o}$s and Tur$\acute{\rm a}$n (see \cite{Et68}). Readers may conf. \cite{Caste10,Dnp13,Nath08} for various results on $\Pr(G)$. In  the following few results we shall compute various energies of the commuting graphs of finite non-abelian groups $G$ such that $\Pr(G) = r$ for some rational number $r$.

\begin{theorem}
Let $G$ be a finite group and $p$ the smallest prime divisor of $|G|$. If $\Pr(G) = \frac{p^2 + p - 1}{p^3}$, then

\[
E(\Gamma_G) =  LE(\Gamma_G) = LE^+(\Gamma_G) = 2(p^2 - 1)|Z(G)| - 2(p + 1).
\]

\end{theorem}
\begin{proof}
If $\Pr(G) = \frac{p^2 + p - 1}{p^3}$, then by \cite[Theorem 3]{dM74}, we have $\frac{G}{Z(G)}$ is isomorphic to ${\mathbb{Z}}_p\times {\mathbb{Z}}_p$. Hence the result follows from  Theorem \ref{main2}.
\end{proof}
As a corollary we have
\begin{corollary}
Let $G$ be a finite group such that $\Pr(G) = \frac{5}{8}$. Then
\[
E(\Gamma_G) =  LE(\Gamma_G) = LE^+(\Gamma_G) = 6|Z(G)| - 6.
\]
\end{corollary}

\begin{theorem}
If $\Pr(G) \in \{\frac{5}{14}, \frac{2}{5}, \frac{11}{27}, \frac{1}{2}\}$, then
$E(\Gamma_G) \in \{11, 7, 6,  3\}$, $LE(\Gamma_G) \in \{\frac{480}{13}, 16, \frac{7}{3}, \frac{16}{5}\}$ and\\
 $LE^+(\Gamma_G) \in \{\frac{370}{13},  6, \frac{16}{5}\}$.
\end{theorem}
\begin{proof}
If $\Pr(G) \in \{\frac{5}{14}, \frac{2}{5}, \frac{11}{27}, \frac{1}{2}\}$, then as shown in \cite[pp. 246]{Rusin79} and \cite[pp. 451]{Nath13}, we have $\frac{G}{Z(G)}$ is isomorphic to one of the groups in $\{D_{14}, D_{10}, D_8, D_6\}$. Hence the result follows from Corollary \ref{main005}.
\end{proof}
%We conclude this section by the following result.
%\begin{theorem}
%If $G$ is a non-solvable group with $\Pr(G) = \frac{1}{12}$ then $G$ is super integral.
%\end{theorem}
%\begin{proof}
%By Proposition 3.3.7 in \cite{Caste10} we have that $G$ is isomorphic to $A_5 \times B$ for some abelian group $B$. Therefore $G$ is an AC-group and hence super integral.
%\end{proof}

Recall that  genus of a graph is the smallest non-negative integer $n$ such that the graph can be embedded on the surface obtained by attaching $n$ handles to a sphere. A graph is said to be planar or toroidal if the genus of the graph is zero or one respectively. In the next two results we compute various energies of $\Gamma_G$ if $\Gamma_G$ is planar or toroidal. We begin with the following lemma.

%\section{More Applications}
\begin{lemma}\label{order16}
Let $G$ be a group isomorphic to any of the following groups
\begin{enumerate}
\item ${\mathbb{Z}}_2 \times D_8$
\item ${\mathbb{Z}}_2 \times Q_8$
\item $M_{16}  = \langle a, b : a^8 = b^2 = 1, bab = a^5 \rangle$
\item ${\mathbb{Z}}_4 \rtimes {\mathbb{Z}}_4 = \langle a, b : a^4 = b^4 = 1, bab^{-1} = a^{-1} \rangle$
\item $D_8 * {\mathbb{Z}}_4 = \langle a, b, c : a^4 = b^2 = c^2 =  1, ab = ba, ac = ca, bc = a^2cb \rangle$
\item $SG(16, 3)  = \langle a, b : a^4 = b^4 = 1, ab = b^{-1}a^{-1}, ab^{-1} = ba^{-1}\rangle$.
\end{enumerate}
Then
\[
E(\Gamma_G) =  LE(\Gamma_G) = LE^+(\Gamma_G) = 18.
\]

\end{lemma}
\begin{proof}
If $G$ is isomorphic to any of the above listed  groups, then $|G| = 16$ and $|Z(G)| = 4$. Therefore, $\frac{G}{Z(G)} \cong {\mathbb{Z}}_2 \times {\mathbb{Z}}_2$. Thus the result follows from Theorem~\ref{main2}.
\end{proof}

%Now we state and proof  the main results of this section.
\begin{theorem}
Let $\Gamma_G$ be the commuting graph of a finite non-abelian  group $G$.  If  $\Gamma_G$  is planar   then  \begin{align*}
E(\Gamma_G) \in & \{3, 6, 7, 12, 18, 26, 30, 76, 17 + 4\sqrt{5} + \sqrt{17}\},\\
LE(\Gamma_G) \in & \{\frac{16}{5}, \frac{7}{3}, 16, 18, \frac{72}{5}, 6, \frac{140}{11}, \frac{504}{19}, \frac{408}{11}, \frac{3924}{59}, \frac{526 + 46\sqrt{13}}{23}\} \quad \text{ and } \\
LE^+(\Gamma_G) \in & \{\frac{16}{5}, 6,   \frac{54}{5}, 18, \frac{256}{11}, \frac{484}{19}, \frac{312}{11}, \frac{3844}{59}, \frac{756}{23}\}.
\end{align*}
\end{theorem}

\begin{proof}
 By  \cite[Theorem 2.2]{AF14}, we have that $\Gamma_G$ is planar if and only if $G$ is isomorphic to either $D_6, D_8, D_{10}$, $D_{12}, Q_8, Q_{12}, {\mathbb{Z}}_2 \times D_8, {\mathbb{Z}}_2 \times Q_8, M_{16}, {\mathbb{Z}}_4 \rtimes {\mathbb{Z}}_4, D_8 * {\mathbb{Z}}_4, SG(16, 3), A_4,$ $A_5, S_4, SL(2, 3)$ or $Sz(2)$.
If $G \cong D_6, D_8, D_{10}$ or $D_{12}$, then by Corollary \ref{main005}, we have
\[
E(\Gamma_G) \in  \{3, 6, 7, 12\}, \quad
LE(\Gamma_G) \in  \{\frac{16}{5}, \frac{7}{3}, 7, \frac{72}{5}\}  \quad \text{ and }\quad
LE^+(\Gamma_G)\in  \{\frac{16}{5}, 6,  \frac{54}{5}\}.
\]
If $G \cong Q_8$ or $Q_{12}$ then,  by Corollary \ref{q4m}, we have
\[
E(\Gamma_G) =  6 \text{ or } 12, \quad
LE(\Gamma_G) =  6 \text{ or } \frac{72}{5} \quad \text{ and }\quad \\
LE^+(\Gamma_G) =  6 \text{ or } \frac{54}{5}.
\]
If $G \cong {\mathbb{Z}}_2 \times D_8, {\mathbb{Z}}_2 \times Q_8, M_{16}, {\mathbb{Z}}_4 \rtimes {\mathbb{Z}}_4, D_8 * {\mathbb{Z}}_4$ or $SG(16, 3)$, then by Lemma \ref{order16}, we have
\[
E(\Gamma_G) = LE(\Gamma_G) = LE^+(\Gamma_G) = 18.
\]
If $G \cong A_4 = \langle a, b : a^2 = b^3 = (ab)^3 = 1\rangle$, then  it can be seen that $\Gamma_{G} = K_3 \sqcup 4K_2$ and so
\[
E(\Gamma_G) =  12, \quad
LE(\Gamma_G) = \frac{140}{11}  \quad \text{ and }\quad
LE^+(\Gamma_G) =  \frac{256}{11}.
\]
If $G \cong Sz(2)$, then by Theorem \ref{order-20}, we have
\[
E(\Gamma_G) =  26,\quad
LE(\Gamma_G) =  \frac{504}{19}  \quad \text{ and }\quad
LE^+(\Gamma_G) =  \frac {484}{19}.
\]
If $G$ is isomorphic to $SL(2, 3)$, then it was shown in the proof of \cite[Theorem 4]{Dutta16} and \cite[Theorem 5.2]{dn016} that
%$\Gamma_G = 3K_2 \sqcup 4K_4$.  Therefore, by Theorem \ref{prethm1}, we have
\[
\spec(\Gamma_G) = \{(-1)^{15}, 1^3, 3^4\}, \L-spec(\Gamma_G) = \{0^7, 2^3, 4^{12}\} \text{ and } \Q-spec(\Gamma_{G}) = \{0^3, 2^{15}, 6^4\}.
\]
Therefore
\[
E(\Gamma_G) =  30, \quad
LE(\Gamma_G) =  \frac{408}{11}  \quad \text{ and }\quad
LE^+(\Gamma_G) =  \frac{312}{11}.
\]
If $G \cong A_5$, then by Proposition \ref{psl}, we have
\[
E(\Gamma_G) =  76, \quad
LE(\Gamma_G) =  \frac{3924}{59}  \quad \text{ and } \quad
LE^+(\Gamma_G) =  \frac{3844}{59}.
\]
noting that $PSL(2, 4) \cong A_5$.

Finally, if $G \cong S_4$, then it was shown in \cite{Dutta16, dn016} that
\[
\spec(\Gamma_G) = \left\lbrace 1^7, (-1)^{10}, (\sqrt{5})^2, (-\sqrt{5})^2, \left(\frac{3 + \sqrt{17}}{2}\right)^1, \left(\frac{3 - \sqrt{17}}{2}\right)^1 \right\rbrace.
\]
\[
\L-spec(\Gamma_G) = \left\lbrace 0^5, 1^3, 2^4, 3^6, 5^1, (4 + \sqrt{13})^2, (4 - \sqrt{13})^2 \right\rbrace.
\]
and
\[
\Q-spec(\Gamma_G) = \left\lbrace 0^4, 1^6, 2^4, 3^3, 5^1, (4 + \sqrt{5})^2, (4 - \sqrt{5})^2,\left(\frac{11 + \sqrt{41}}{2}\right)^1, \left(\frac{11 - \sqrt{41}}{2}\right)^1 \right\rbrace.
\]
Therefore,
\[
E(\Gamma_G) =  17 + 4\sqrt{5} + \sqrt{17}, \quad
LE(\Gamma_G) =  \frac{526 + 46\sqrt{13}}{23}  \quad \text{ and } \quad
LE^+(\Gamma_G) =  \frac{756}{23}.
\]
This completes the proof.
\end{proof}

%In \cite[Theorem 2.3]{AF14}, Afkhami et al. have classified all finite non-abelian groups whose commuting graphs are toroidal. Unfortunately, the  statement of Theorem 2.3 in \cite{AF14} is printed incorrectly.  We list the correct version of \cite[Theorem 2.3]{AF14} below, since we are going to use it.
%\begin{theorem}\label{toroidal}
%Let $G$ be a finite non-abelian group. Then $\Gamma_G$ is toroidal if and only if $\Gamma_G$ is projective if and only if $G$ is isomorphic to either $D_{14}, D_{16}, Q_{16}$, $QD_{16},   D_6 \times {\mathbb{Z}}_3,   A_4 \times {\mathbb{Z}}_2$ or ${\mathbb{Z}}_7 \rtimes {\mathbb{Z}}_3$.
%\end{theorem}

\begin{theorem}
Let $\Gamma_G$ be the commuting graph of a finite non-abelian  group $G$.  If  $\Gamma_G$    is toroidal, then
\begin{align*}
E(\Gamma_G)\in & \{ 11, 18, 42, 25, 22, 34\},\\
LE(\Gamma_G)\in & \{\frac{480}{13}, \frac{230}{7}, \frac{962}{15}, \frac{236}{7}, \frac{103}{4}, 48, \frac{390}{11}\} \quad \text{ and } \\
LE^+(\Gamma_G) \in & \{\frac{370}{13}, \frac{192}{7}, 185, \frac{480}{7}, \frac{677}{20}, 59, \frac{408}{11}\}
\end{align*}
\end{theorem}

\begin{proof}
By Theorem 6.6 of \cite{das13},  we have   $\Gamma_G$ is toroidal if and only if $G$ is isomorphic to either $D_{14}, D_{16}$, $Q_{16}, QD_{16},   D_6 \times {\mathbb{Z}}_3,   A_4 \times {\mathbb{Z}}_2$ or ${\mathbb{Z}}_7 \rtimes {\mathbb{Z}}_3$.
If $G \cong D_{14}$ or $D_{16}$ then, by Corollary \ref{main005}, we have
\[
E(\Gamma_G) =  11 \text{ or } 18, \quad
LE(\Gamma_G) =  \frac{480}{13} \text{ or } \frac{230}{7}  \quad \text{ and } \quad
LE^+(\Gamma_G) =  \frac{370}{13} \text{ or } \frac{192}{7}.
\]
If $G \cong Q_{16}$ then,  by Corollary \ref{q4m}, we have
\[
E(\Gamma_G) =  42, \quad
LE(\Gamma_G) =  \frac{962}{15}  \quad \text{ and } \quad
LE^+(\Gamma_G) =  185.
\]
If $G \cong QD_{16}$ then, by Proposition \ref{semid}, we have
\[
E(\Gamma_G) =  18, \quad
LE(\Gamma_G) =  \frac{236}{7}  \quad \text{ and } \quad
LE^+(\Gamma_G) =  \frac{480}{7}.
\]
If $G \cong {\mathbb{Z}}_7 \rtimes {\mathbb{Z}}_3$ then, by Proposition \ref{order-pq}, we have
\[
E(\Gamma_G) =  25, \quad
LE(\Gamma_G) =  \frac{103}{4}  \quad \text{ and } \quad
LE^+(\Gamma_G) =  \frac{677}{20}.
\]
If $G$ is isomorphic to  $D_6\times {\mathbb{Z}}_3$, then
\[
E(\Gamma_G) =  22, \quad
LE(\Gamma_G) =  48  \quad \text{ and } \quad
LE^+(\Gamma_G) =  59.
\]
Finaly, if $A_4\times {\mathbb{Z}}_2$, then
\[
E(\Gamma_G) =  34, \quad
LE(\Gamma_G) =  \frac{390}{11}  \quad \text{ and } \quad
LE^+(\Gamma_G) =  \frac{408}{11}.
\]
This completes the proof.
\end{proof}

%\section{Some more results}

%\begin{acknowledgements}
%If you'd like to thank anyone, place your comments here
%and remove the percent signs.
%\end{acknowledgements}

% BibTeX users please use one of
%\bibliographystyle{spbasic}      % basic style, author-year citations
%\bibliographystyle{spmpsci}      % mathematics and physical sciences
%\bibliographystyle{spphys}       % APS-like style for physics
%\bibliography{}   % name your BibTeX data base

% Non-BibTeX users please use

\end{document}